\documentclass[12pt]{amsart}
\usepackage{graphicx}
\usepackage{amscd}
\usepackage{epsf}
\usepackage{graphicx}
\usepackage{pstricks}

\newtheorem{thm}{Theorem}[section]

\newtheorem{defn}[thm]{Definition}
\newtheorem{lemma}[thm]{Lemma}

\newtheorem{cor}[thm]{Corollary}
\newtheorem{remark}[thm]{Remark}
\newtheorem{example}[thm]{Example}

\usepackage{amsmath}
\usepackage{amsxtra}
\usepackage{amscd}
\usepackage{amsthm}
\usepackage{amsfonts}
\usepackage{amssymb}
\usepackage{eucal}
\newcommand{\bmb}{\left( \begin{array}{rr}}
\newcommand{\enm}{\end{array}\right)}

\newcommand{\g}{{\mathfrak{g}}}
\newcommand{\h}{{\mathfrak{h}}}

\newcommand{\n}{{\mathfrak{n}}}
\newcommand{\qY}{\mathbf Y}
\newcommand{\p}{p}

\newcommand{\bl}{{\boldsymbol\ell}}
\newcommand{\bq}{{\mathbf q}}

\renewcommand{\sl}{{\mathfrak{sl}}}
\newcommand{\KR}{{{\rm KR}}}

\newcommand{\C}{{\mathbb C}}
\newcommand{\Z}{{\mathbb Z}}

\newcommand{\N}{{\mathbb N}}
\newcommand{\Hom}{{\rm Hom}}

\newcommand{\ba}{{\mathbf a}}

\newcommand{\bp}{{\mathbf p}}
\newcommand{\bm}{{\mathbf m}}
\newcommand{\bn}{{\mathbf n}}
\newcommand{\bx}{{\mathbf x}}
\newcommand{\by}{{\mathbf y}}

\newcommand{\bj}{{\mathbf j}}

\newcommand{\bX}{{\mathbf X}}

\newcommand{\bz}{{\mathbf z}}

\newcommand{\al}{{\alpha}}
\newcommand{\qQ}{\widehat Q}
\newcommand{\bqQ}{\widehat{\mathbf Q}}
\newcommand{\CT}{{\rm CT}}

\newcommand{\qx}{\mathbf X}
\newcommand{\bell}{\boldsymbol\ell}
\newcommand{\qX}{{X}}

\newcommand{\qbin}[2]
{{
\left[
\begin{matrix}{\displaystyle #1}\\
{\displaystyle #2}\end{matrix}
\right]
}}

\numberwithin{equation}{section}

\begin{document}
\title{Quantum cluster algebras and fusion products}
\author{Philippe Di Francesco} 
\address{
Institut de Physique Th\'eorique du Commissariat \`a l'Energie Atomique, 
Unit\'e de Recherche associ\'ee du CNRS,
CEA Saclay/IPhT/Bat 774, F-91191 Gif sur Yvette Cedex, 
FRANCE. e-mail: philippe.di-francesco@cea.fr}
\author{Rinat Kedem}
\address{Department of Mathematics, University of Illinois, Urbana, IL 61821 USA. e-mail: rinat@illinois.edu}

\begin{abstract}
$Q$-systems are recursion relations satisfied by the characters of the restrictions 
of special finite-dimensional modules of quantum affine algebras. They can also be viewed as
mutations in certain cluster algebras, which have a natural quantum deformation.
In this paper, we explain the relation in the simply-laced case between the resulting quantum
$Q$-systems and the graded tensor product of Feigin and Loktev. 
We prove the graded version of the $M=N$ identities, and
write expressions for these as  non-commuting evaluated multi-residues
of suitable products of solutions of the quantum $Q$-system. This leads to
a simple reformulation of Feigin and Loktev's fusion coefficients as matrix elements in a representation 
of the quantum $Q$-system algebra.
\end{abstract}

\maketitle
\date{\today}

\section{Introduction}
The graded tensor product, or fusion product, of Feigin and Loktev \cite{FL} is a refinement, 
or $\g$-equivariant grading, of the tensor product of current algebra modules.
In \cite{AK,DFK} it was proved that the Hilbert polynomial $\mathcal M(q)$ of the graded multiplicity 
space is equal to the fermionic formula $M(q)$, introduced by \cite{KR,HKOTY}. 
This was done by proving a conjecture of \cite{HKOTY}, that the ungraded formula $M(1)$ is equal to an unrestricted sum,  $N(1)$. 
The latter is proven to be equal \cite{HKOTY}  to the multiplicity of the irreducible 
component in the tensor product of representations. 
Together with
certain inequalities \cite{AK} and the manifest positivity of the sum $M(q)$,
our proof \cite{DFK} of the equality of ungraded sums, $M(1)=N(1)$, implies that
$M(q)=\mathcal M(q)$, the Feigin-Loktev conjecture. This proof uses crucially
a polynomiality property the solutions of a system of recursion relations known as the $Q$-system \cite{KR}, which is satisfied by the characters of special  so-called Kirillov-Reshetikhin modules \cite{Nakajima}.

In the present paper, we prove the identity of the fermionic sums when $q\neq 1$, i.e. 
a set of identities of the form $M(q)=N(q)$. 
This is a natural extension of our proof for the ungraded case: As shown in \cite{Ke08}, 
the $Q$-systems may be 
embedded into certain cluster algebras \cite{FZ}, 
which have natural quantum deformations \cite{BZ}.
Such a deformation gives rise to a natural definition of quantum $Q$-system, 
a set of recursion relations satisfied by non-commuting elements \cite{DFK10}. 
We show that this $q$-deformation gives precisely the grading of the Feigin-Loktev fusion product. 
Moreover, quantum cluster algebras have many properties analogous to those of
commutative cluster algebras: In particular, the Laurent property, and the consequent 
polynomiality of solutions of the $Q$-system with particular boundary conditions have a quantum counterpart. 
This allows us to $q$-deform the theorems and proofs of \cite{DFK} to show the equality of the 
polynomials $M(q)$ and $N(q)$, and to completely characterize the Feigin-Loktev fusion product
within the framework of the quantum $Q$-system algebra. 
We need a few definitions in order to state our precise result. 

In their solution of the generalized, inhomogeneous Heisenberg spin
chain corresponding to the Yangian $Y(\g)$, where $\g$ is a simple Lie
algebra, Kirillov and Reshetikhin \cite{KR} introduced a conjectural
formula for {graded} multiplicities of the irreducible $\g$-modules in
the tensor product of finite-dimensional Yangian modules \cite{HKOTY}. The modules
are of Kirillov-Reshetikhin (KR) type \cite{Chari}, denoted by
$\KR_{\al,i}(z)$, where $i \omega_\al$ is the highest weight of the
KR-module with respect to $\h\subset\g\subset Y(\g)$, $\omega_\al$ a fundamental
weight of $\g$, and $\al\in I_r$, the set of root indices. The
variable $z\in \C^*$ is called the spectral parameter.
Denote the graded multiplicites by
$M_{\lambda,\bn}(q)$\footnote{In this paper, the
  notation differs from that of \cite{HKOTY} by
  $q\to q^{-1}$. That is,
  $M_{\lambda,\bn}(q)$ here is a polynomial in $q$ rather than $q^{-1}$.}
and define the $q$-graded
character to be the following:
$$
{\rm char}_{q}\ {\rm Res}^{Y(\g)}_\g\left(\underset{{\al,i}}{\otimes}\
  \KR_{\al,i}^{\otimes n_{\al,i}}\right)= \sum_{\lambda\in P^+}
M_{\lambda,\bn}(q) {\rm char}(V_\lambda),
$$
where the (omitted) spectral parameters of the KR-modules are
assumed to be pairwise distinct, and $V_\lambda$ is the irreducible
$\g$-module with highest weight $\lambda$.  Various definitions of the
grading will be discussed below.

Kirillov and Reshetikhin conjectured that the multiplicity
$M_{\lambda,\bn}(1)$ is equal to the number of Bethe eigenvectors,
which they were able to compute. This is called the completeness
conjecture. In its graded form, their formula is
\begin{equation}\label{msumintro}
  M_{\lambda,\bn}(q^{-1}) = \underset{m_{\al,i}\in\Z_+}{{\sum}^{(1),(2)}}
q^{Q(\bn,\bm)} \prod_{\al,i} \qbin{p_{\al,i}+m_{\al,i}}{m_{\al,i}}_q,
\end{equation}
where the sum is taken over vectors with non-negative integers entries
$\bm=(m_{\al,i})_{\al\in I_r, i\in \N}$, $I_r=\{1,2,...,r\}$, $r={\rm rk}(\g)$, and the sum is taken under
the restrictions (1) and (2) involving the integers $p_{\al,i}$
as follows. Let $C$ be the Cartan matrix of $\g$ and $A$ be the matrix
with entries $[A]_{i,j} = \min(i,j)$, we define the vectors
$\bp=(p_{\al,i})$ and the quadratic form $Q$ by:
$$
\bp = (\mathbb I \otimes A)\bn - (C\otimes A)\bm,\qquad 
Q(\bn,\bm) = \frac{1}{2}\bm^t(C\otimes A)\bm -\bm^t (\mathbb I \otimes A)\bn.
$$
The restrictions on the sum in \eqref{msumintro} are as follows: 
\begin{enumerate}
\item
$p_{\al,i}\geq 0$ for all $\al,i$;
\item $\lim_{i\to\infty} \sum_\al p_{\al,i} \omega_\al= \lambda$.
\end{enumerate}
Restriction (1) is non-trivial because the $q$-binomial coefficient
$$
\qbin{p+m}{m}_q\overset{\rm def}{=}
\frac{(q^{p+1};q)_\infty(q^{m+1};q)_\infty}{(q;q)_\infty(q^{p+m+1};q)_\infty}
  \qquad (p;q)_\infty=\prod_{j\geq 0}(1-q^j p)
$$
does not necessarily vanish, and may be negative, if $p<0$.

In \cite{DFK} we proved that the
restriction (1), $p_{\al,i}\geq 0$, does not affect the value of the
integers $M_{\bn,\lambda}(1)$. (The proof of \cite{DFK} covers the non
simply-laced case as well.) .
The sum in \eqref{msumintro} without the
restriction (1) was called the $N$-sum in \cite{HKOTY}:
\begin{equation}\label{nsumintro}
  N_{\lambda,\bn}(q^{-1}) \overset{\rm def}{=} \underset{m_{\al,i}\in\Z_+}{{\sum}^{(2)}}
q^{Q(\bn,\bm)} \prod_{\al,i} \qbin{p_{\al,i}+m_{\al,i}}{m_{\al,i}}_q.
\end{equation}
In fact, the ungraded form of this sum, at $q=1$, first appeared in
the case of the tensor product of the fundamental representations of $\sl_2$,
in the original solution of Hans Bethe \cite{Bethe}.  Note that the
$N$-sum is not manifestly positive.  A positive sum of terms which are
products of binomial coefficients is generally referred to as a
fermionic formula.

In \cite{HKOTY}, the authors proved that if the characters of
KR-modules satisfy the $Q$-system, then the multiplicity of the irreducible
component in the tensor product of these modules is equal to the
$N$-sum \eqref{nsumintro}. In the simply-laced case, Nakajima \cite{Nakajima}
proved that the $q$-characters of these modules satisfy the $T$-system
\cite{KR,KNS}, which implies the $Q$-system. This shows that
$N_{\lambda,\bn}(1)$ is the multiplicity of the irreducible
$\g$-modules in the tensor product. Given the results in \cite{DFK},
it is also equal to $M_{\lambda,\bn}(1)$, and the completeness of the
Bethe ansatz follows. The proof in \cite{DFK} was the original
inspiration for our formulation of 
$Q$-systems \cite{Ke08} and $T$-systems \cite{DFK2} as cluster
algebras \cite{FZ}.

In this paper, we will give a new interpretation for the grading in the sums
in \eqref{msumintro} and \eqref{nsumintro}, involving quantum cluster
algebras \cite{BZ}, and use this connection to
prove that $M_{\lambda,\bn}(q^{-1})=N_{\lambda,\bn}(q^{-1})$ for all $q\in\C^*$.

There are several interpretations of the grading of the tensor
product. The grading of the tensor product of Yangian modules given by
the polynomial $M_{\lambda,\bn}(q)$ was defined by Kirillov and
Reshetikhin by inspiration from the form of the Bethe ansatz
solutions. Subsequently, another grading was introduced \cite{OSS} for
tensor products of $\KR$-modules of certain quantum affine algebras
$U_q(\widehat{\g})$ by using an energy function on tensor products of
crystal bases. Finally, the tensor product of the KR-modules defined
for the algebra $\g[t]$ \cite{ChariMoura} admits yet another grading
\cite{FL}, compatible with the homogeneous grading in $t$ of
$U(\g[t])$. It turns out that these definitions all give the same
graded decomposition coefficients for the three algebras $Y(\g)$,
$U_q(\g)$, $U(\g[t])$ \cite{Ke11}.

As noted above, the expression \eqref{nsumintro} for the multiplicities
$N_{\lambda,\bn}(1)$ follows from the fact that the
characters of KR-modules obey a functional relation called the
$Q$-system.  In the case of simply-laced Lie algebras
which we treat in this paper, the $Q$-system has the form \cite{KR,KNS}
\begin{equation}\label{qsyssl} 
  Q_{\al,n+1}Q_{\al,n-1}=Q_{\al,n}^2-\prod_{\beta\neq \al} Q_{\beta,n}^{-C_{\al,\beta}},\qquad \al\in I_r,n \in \Z.
\end{equation}
If one sets $Q_{\al,0}=1$ for all $\al\in I_r$, and $Q_{\al,1}={\rm
  char}\ \KR_{\al,1}$, then it follows from the work of Nakajima
\cite{Nakajima} that the solution $Q_{\al,i}$ with $i>1$ of
\eqref{qsyssl} is equal to the character of $\KR_{\al,i}$.

In \cite{DFK}, we expressed the multiplicity $M_{\lambda,\bn}(1)$ as
a residue in $\{Q_{\al,1}\}_{\al\in I_r}$ of the product (with
sufficiently large $k$) $\prod_\al
\left(Q_{\al,k}/Q_{\al,k+1}\right)^{\ell_\al+1}\prod_{\al,i}Q_{\al,i}^{n_{\al,i}}$
of solutions of the $Q$-system \eqref{qsyssl} evaluated at
$\{Q_{\al,0}=1\}_{\al\in I_r}$.  Here, $\ell_\al$ is defined by
$\lambda = \sum_\al\ell_\al \omega_\al$.

In this paper, we follow an analogous argument in the quantized case, to
define another interpretation for the grading of the multiplicities
$M_{\lambda,\bn}(q)$ in terms of the quantum $Q$-system, introduced in
\cite{DFK10} as a non-commutative deformation of the $Q$-system based
on its connection with cluster algebra.

As shown in \cite{Ke08}, each of the $Q$-system relations
\eqref{qsyssl} is a mutation in a certain cluster algebra. A
useful corollary \cite{DFK2} to the Laurent property of cluster algebra is the
polynomiality  for the solutions of the $Q$-system
under the evaluation at $\{Q_{\al,0}=1: \al\in I_r\}$ as a function of
$\{Q_{\al,1}: \al\in I_r\}$. 
This polynomiality property was a crucial ingredient in the proof of the $M=N$ identity \cite{DFK}, that
is, it implies that the restriction (1) on the sum in \eqref{msumintro} is unnecessary.

Similarly, applying the definition of quantum cluster algebras due to
Berenstein and Zelevinsky \cite{BZ}, one may define the quantum
$Q$-system in the same manner as we did in \cite{DFK10} for the $A$ cases.  
The main result of this paper is a
non-commutative version of the results of \cite{DFK}, for the graded
multiplicities $M_{\lambda,\bn}(q^{-1})$.  By use of a suitable
generating function with non-commutative arguments, the graded multiplicities may
be expressed as an evaluated residue of a product of non-commutative elements
$\qQ_{\al,i}$'s, where $\{\qQ_{\al,i}\}$ are solutions of the quantum $Q$-system with generic
initial conditions.
The ``evaluation at $\qQ_{\beta,0}=1$" must be defined appropriately, 
as $\qQ_{\beta,0}$ does not commute with the other $\qQ$'s; also,
the notion of multi-residue must be adapted to the non-commutation of the variables 
involved. 

We show that the correct evaluated multi-residue amounts
to computing a matrix element in a certain infinite-dimensional
representation of the quantum $Q$-system algebra. 
As in the commutative case, we obtain an analog
of the polynomiality lemma of Ref. \cite{DFK2}: Any matrix element of a
product of quantum $Q$-system solutions can be written as a matrix element of a
{\it polynomial} in the variables $\{\qQ_{\al,1}\}_\al$. 
This allows us to prove the graded version of the $M=N$ theorem in the simply-laced case.

Our purpose here is to give a concrete application of the quantum $Q$-system,
which we introduced in \cite{DFK10} and solved in terms of paths on
graphs with $q$-commuting weights, in the case of type $A$, giving rise to 
compact $q$-series generating functions for basic cluster variables \cite{DF11}. 
This paper is restricted to the simply-laced case, but the extension
to the non-simply-laced case should follow the same steps as those given in the commutative case \cite{DFK} using the techniques of the present paper.
In particular non-simply-laced quantum $Q$-systems can be defined via the
quantum versions of the associated cluster algebras given in \cite{DFK2}.

We note that a more geometric representation-theoretical interpretation of a special case of the 
$M$-sum was conjectured by Lusztig \cite{Lusztig}, and proved in \cite{KN}. 
In fact, the result of \cite{KN} implies the graded $M=N$ conjecture in the special 
case involving only fundamental modules, for simply-laced $\g$.

The paper is organized as follows. In Section 2 we review the
definition of the Feigin-Loktev fusion product of KR-modules, and in
Section 3, of the quantum $Q$-system. In Section 4, we give the
detailed explanation of how to obtain the graded $M=N$ theorem for the case of $A_1$
(Theorem \ref{mequalsn}), using
the $A_1$ quantum $Q$-system. We give an explicit expression for $M_{\lambda,\bn}(q)$
in terms of suitably evaluated 
multi-residues involving solutions of the quantum $Q$-system (Theorem \ref{compum}), 
and show how to represent these graded multiplicities
as matrix elements of an infinite-dimensional representation of the corresponding quantum $Q$-system algebra
(Theorem \ref{oscirepone}). 
The other simply-laced cases are very
similar although notationally more complicated, and are treated in Section 5.
The generalized versions of Theorems \ref{mequalsn}, \ref{compum} and \ref{oscirepone}
are Theorems \ref{M=Nsl}, \ref{howtocomputeM}, and \ref{oscirepsl} respectively.

\vskip.1in
\noindent{\bf Acknowledgements.} This work is supported by the CNRS PICS grant 05859. The work of RK on is supported by NSF grants DMS-0802511 and DMS-1100929. PDF acknowledges support from ANR grant GranMa. The authors thank MSRI for hosting part of this research during the program ``Random Matrix Theory, Interacting Particle Systems and Integrable Systems". RK also thanks CEA IPhT-Saclay for its hospitality and support while this research was conducted. We thank H. Nakajima and N. Reshetikhin for useful discussions.

\section{Graded tensor products}
We give here one of the representation-theoretical definitions of the
graded tensor product multiplicities, introduced Feigin and Loktev
\cite{FL} for graded $\g[t]$-modules. The authors referred to this
product as the fusion product, due to its origin in the fusion product
of conformal field theory.

Feigin and Loktev conjectured that their construction was related to
the graded multiplicities in Equation \eqref{msumintro}. It was proved
in \cite{AK} that in the case of the tensor product of
KR-modules of any simple Lie algebra, the graded
multiplicities were bounded from above by the polynomials
\eqref{msumintro} (generalized to the case of non simply-laced
$\g$). Due to the construction used in \cite{AK}, if the value of this
multiplicity at $q=1$ is equal to the usual tensor product
multiplicity, then the equality holds between graded
multiplicities. It was known \cite{HKOTY} that the unrestricted sum
$N_{\lambda,\bn}(1)$ was equal to the tensor product multiplicity,
given the $Q$-system identity for characters of KR-modules, proved in
\cite{Nakajima,Hernandez}. Thus, the proof in \cite{DFK} that the
unrestricted $N$-sum is equal to the $M$-sum at $q=1$ implies the Feigin-Loktev
conjecture.

\subsection{Graded cyclic modules}
Let $\g$ be a simple Lie algebra with Cartan decomposition 
$$\g\simeq
\n_- \oplus \h \oplus \n_+.$$
 Let $t$ be a formal variable and define the current algebra, $\g[t]:=\g\otimes \C[t]$. This is the Lie algebra with basis
$$
\{x[n]:=x\otimes t^n,\ x\in \g,
n\in\Z_+\}
$$
and relations inherited from $\g$:
$$
[x\otimes t^n, y\otimes t^m]_{\g[t]} = [x,y]_\g \otimes t^{m+n}, \quad x,y\in\g.
$$ 
This algebra is graded by degree in $t$, as is its universal enveloping algebra $U(\g[t])$.

Let $V$ denote a $\g[t]$-module on which $\g[t]$ acts via some representation $\pi$.
We introduce a ``translated action" of $\g[t]$ on $V$ as follows.  Let
$z\in\C^*$ be a non-zero complex number. The representation $\pi_z$ on
the $\g[t]$-module $V$ is given by 
$$\pi_z(x\otimes t^n) w = \pi(x\otimes (t+z)^n) w=\sum_{j=0}^n
{n\choose j} z^{n-j} \pi( x\otimes t^j) w, \quad x\in\g, w\in V.$$
Note that the translated action of $\g[t]$ does not preserve the
grading with respect to degree in $t$, as $z$ is a complex
number. However, the shifted action of $U(\g[t])$ is still filtered by
degree in $t$ as the top graded component of the action of $x\otimes
t^n$ is $n$.

Assume $V$ is a cyclic $\g[t]$-module with respect to the translated
action $\pi_z$. Let $v$ be the cyclic vector, $V\simeq\pi_z(U(\g[t]))
v.$ We assign the degree 0 to $v$.  Let $U^{(\leq n)}$ denote elements
in $U(\g[t])$ of total degree less than or equal to $n$ in $t$. This
gives a filtration of $U{(\g[t])}$, with $U^{(\leq n)}\subset U^{(\leq
n+1)}$, and $\pi_z(U(g[t])v$ inherits this filtration. We define the graded module $\overline{V}$
to be the associated graded space of this filtration of $\pi_z(U(\g[t]))v$. 
The graded components are
$\g$-modules, as $\g\subset\g[t]$ has degree 0.

KR-modules for $\g[t]$ are of the type described
above \cite{ChariMoura,AK}. Denote the generator $x\otimes t^n$ of
$\g[t]$ by $x[n]$. The KR-module $\KR_{\al,i}(z)$ is generated by the
cyclic vector $v$, a highest weight vector with respect to $\g$, with
highest weight of the form $i \omega_\al$. The action of $\g[t]$ is
defined as follows. Given $z\in \C^*$, define the representation
$\pi_z$ of $\g[t]$ as the quotient of the action of $U(\g[t])$ on $v$
by the relations $\pi_z(x[n]) v=0\ (x\in \n_+, n\geq 0)$,
$\pi_z(f_\beta[n]) v=0\ (n\geq \delta_{\alpha,\beta})$ and
$\pi_z(h[n]) v=0$ $(h\in\h, n>0)$, and $\pi_z(f_\al[0]^{i+1})v=0$,
where $f_\al$ is the element of $\n_-$ corresponding to the root
$-\alpha$. The KR-module is the associated graded space of
$\pi_z(U(\g[t]))v$.

If $\g\neq A_r$, the restriction of the KR-module to $\g[t]$ is not
necessarily irreducible. In general, $\KR_{\al,i}(z)\simeq
V_{i\omega_\al}\oplus t \N[t]\times$``smaller modules" as a
$\g$-module, where the smaller modules have highest weights which are
strictly less than $i\omega_\al$ in the dominance ordering. KR-modules are
finite-dimensional, and can be thought of as the ``smallest"
evaluation modules of $\g[t]$ containing $V_{i\omega_\al}$ which have a
deformation to the quantum affine algebra. It is a theorem that the
dimension of the graded KR-module is independent of the parameter $z$.

\subsection{Tensor products}
One can repeat the same construction with tensor products to define a graded tensor product. Choose $N$ cyclic $\g[t]$-modules $\{V_1,...,V_N\}$ with cyclic vectors $\{v_1,...,v_N\}$ and
$N$ pairwise distinct complex numbers $\{z_1,...,z_N\}$. The translated co-product action $\Delta^{(N)}_\bz$ on the tensor product $V_1\otimes V_2\otimes \cdots \otimes V_N$
of representations localized at $z_i$ is naturally defined to be
$$
\Delta^{(N)}_\bz(x\otimes f(t)) = \sum_{i=1}^N \pi_{z_i}^{(i)}(x\otimes f(t))=\sum_{i=1}^N \pi^{(i)}(x\otimes f(t+z_i))
$$
where $\pi^{(i)}$ denotes the representation $\pi$ acting on the $i$-th factor in the tensor product.

We let $U(\g[t])$ act on the tensor product of cyclic vectors
$v_1\otimes \cdots \otimes v_N$ via this co-product. We choose the
tensor product of cyclic vectors to have degree 0. The resulting space
is isomorphic (as a vector space) to the tensor product of
$\g[t]$-modules. Moreover, it is filtered by degree in $t$ and we can
take the associated graded space. This graded $\g$-module is called
the Feigin-Loktev fusion product \cite{FL} of the modules
$\{V_1,...,V_N\}$, denoted by $\mathcal F^*_{\{V_1,...,V_N\}}$. Its
graded components are again $\g$-modules. In principal, the fusion
product depends on the parameters $z_i$.

By definition, the fusion product is commutative. If we choose all
modules $V_i$ to be of KR-type, it was conjectured by Feigin and
Loktev and proven by \cite{AK,DFK} that the fusion product is
independent of the choice of complex numbers $z_i$.  We can therefore
parameterize the fusion product by a vector $\bn=(n_{\al,i})$ where
$n_{\al,i}$ is the number of KR-modules of type $\KR_{\al,i}(z)$ in the
product, for some $z$. Let us denote this fusion product by
$\mathcal F^*_\bn$.

\subsection{Graded multiplicities}
As noted above, the graded components of the fusion product are
$\g$-modules. Let $\mathcal F^*_\bn[m]$ denote the graded component of
degree $m$. Define the generating function for multiplicities of the
irreducible components $V(\lambda)$ in $\mathcal F^*_\bn$ to be
\begin{equation}\label{gramult}
\mathcal M_{\lambda,\bn}(q) = \sum_{m\geq 0} \dim \Hom_\g(\mathcal
F^*_\bn[m],V(\lambda)) q^m.
\end{equation}

One can prove that the fusion product of KR-modules is
independent of the localization parameters $z_i$
by computing these graded multiplicities explicitly. This was done in
\cite{AK,DFK} for any simple Lie algebra, and for any set of
KR-modules, where it was shown that $\mathcal
M_{\lambda,\bn}(q)=M_{\lambda,\bn}(q)$, as a consequence of
the identity $\mathcal M_{\lambda,\bn}(1)=M_{\lambda,\bn}(1)$ and
of the positivity of the graded sums.

To prove that $\mathcal M_{\lambda,\bn}(1)=M_{\lambda,\bn}(1)$, it was
necessary to use the theorems of \cite{HKOTY,Nakajima,Hernandez} for
$N_{\lambda,\bn}(1)$ and then prove that
$M_{\lambda,\bn}(1)=N_{\lambda,\bn}(1)$ \cite{DFK}. The latter follows
from the polynomiality of the solutions of the $Q$-system as a
function of the characters of the fundamental KR-modules.

There are two ways to understand this polynomiality of the solutions
of the $Q$-system under the evaluation at $Q_{\al,0}=1$. One is via
representation theory, by
using the properties of the
Groethendieck ring generated by the
fundamental KR-modules.
The other is combinatorial, and borrows from the connection to cluster algebras,
as explained above.
%

The quantum version of this connection will allow us in particular to prove 
the graded version of the $M=N$
conjecture in the simply-laced case, namely that $M_{\lambda,\bn}(q)=N_{\lambda,\bn}(q)$ 
(Theorem \ref{M=Nsl} below), 
as a consequence of a non-commutative version of the polynomiality
property for $Q$-system solutions when specialized to KR characters.
We will also obtain 
an explicit expression of the graded $M$ or $N$ coefficients as matrix elements 
of a representation of the ``non-negative" part of the quantum $Q$-system
algebra.

%
%
%
\section{The quantum $Q$-system}
\subsection{Simply-laced $Q$ systems and cluster algebras}

\newcommand{\vm}{{\overset{\rightarrow}{m}}}

Any solution of the $Q$-system \eqref{qsyssl} can be expressed as a function of initial data 
consisting of $2r$ elements. An example is what we call the {\em fundamental initial data}, 
which consists of the components of the vector
\begin{equation}\label{fid}
\by_0=(Q_{\al,0},Q_{\al,1})_{\al\in I_r}.
\end{equation}
More generally, a valid set of initial data is determined by a generalized Motzkin path:
\begin{lemma}
The components of any vector of the form
\begin{equation}\label{admissible}
\by_{\vm} = (Q_{\al,m_\al},Q_{\al,m_{\al}+1})_{\al\in I_r},\qquad \vm=(m_1,...,m_r), 
\ |m_\al - m_{\beta}|\leq 1\ {\rm whenever}\ C_{\al,\beta}=-1.
\end{equation}
constitute a valid set of initial data for the $Q$-system: Any solution $Q_{\al,n}$ can be expressed as a 
Laurent polynomial in these variables.
\end{lemma}
The vector $\vm$ in \eqref{admissible} is a generalized Motzkin path associated with the Cartan matrix $C$.

The rank $2r$ cluster algebras \cite{FZ} corresponding to the Q-systems \eqref{qsyssl}
were introduced in \cite{Ke08}. Cluster algebras have an exchange relation without subtractions, 
whereas the $Q$-system as written is an exchange relation with a subtraction. Therefore, in order to 
avoid the use of coefficients (as in the Appendix of \cite{DFK2}) in the cluster algebra, the variables 
$Q_{\al,n}$ can be rescaled:
Let $x_{\al,n}=\exp{( i \pi \sum_\al [C^{-1}]_{\beta,\al})}Q_{\al,n}$. These satisfy \eqref{qsyssl} with the 
minus sign on the right hand side changed to a plus sign:
$$
x_{\al,n+1}x_{\al,n-1}=x_{\al,n}^2+\prod_{\beta\neq\al}x_{\beta,n}^{-C_{\beta,\al}}.
$$
 The cluster algebra is defined from 
the initial cluster $(\bx_0, B)$, where
$\bx_0=(x_{\al,0},x_{\al,1})_{\al\in I_r}$, and  the exchange matrix $B$ is
\begin{equation}\label{bmatrix}
B=\begin{pmatrix} 0 & -C \\ C & 0\end{pmatrix}.
\end{equation}
The $Q$-system equations are a subset of the mutations in this cluster algebra: Those which lead  
to clusters involving  only admissible data  of the form \eqref{admissible}.

%

\subsection{Quantum $Q$-systems}
The quantum Q-system for $A_r$ was introduced in \cite{DFK10}, inspired by the definition \cite{BZ} 
of quantum cluster algebras. Here we extend this definition to the simply-laced case. 

In the simplest case, a quantum cluster algebra of rank $n$ corresponding to the exchange matrix 
$B$ is the non-commutative algebra generated by the variables $\bX=(X_1,...,X_n)$ with relations
\begin{equation}\label{commuq}
X_i X_j = t^{\Lambda_{i,j}} X_j X_i,
\end{equation}
where $\Lambda_{i,j}$ are the entries of a matrix $\Lambda$ which satisfies a ``compatibility relation" 
with the exchange matrix $B$:\footnote{Since $B$ is skew-symmetric, the difference between our 
compatibility relation and that of \cite{BZ}, which uses $B^T$, corresponds simply to the changing $q\to q^{-1}$.}
$$
B \Lambda = \delta \mathbb I, \qquad \delta\in \N.
$$
Defining $\bX^{\ba}:= t^{\frac{1}{2} \sum_{i>j} \Lambda_{i,j} a_i a_j} X_1^{a_1}\cdots X_n^{a_n}$, 
mutations $\mu_i$ act in the same way on the matrix $B$ as in the case of the classical cluster 
algebra. We write $\mu_i(X_k)=X_k$ if $i\neq k$ and
\begin{equation}\label{mutaX}
\mu_i (X_i) = \bX^{\mathbf b_+[i]} + \bX^{\mathbf b_-[i]},\qquad (\mathbf b_{\pm}[i])_j 
:= ([\pm B]_+ - \mathbb I)_{i,j}.
\end{equation}

The quantum cluster algebra associated to the exchange matrix $B$ in \eqref{bmatrix} is a 
non-commutative algebra generated by the variables $\qx_0=(\qX_{\al,0},\qX_{\al,1})_{\al\in I_r}$ 
subject to the following commutation relations:
$$
\qX_{\al,n}\qX_{\beta,m} = t^{\lambda_{\al,\beta}(m-n)} 
\qX_{\beta,m}\qX_{\al,n}, \quad m,n\in \{0,1\}.
$$
where 
 $\lambda_{\al,\beta}$ are the elements of the matrix 
 \begin{equation}\label{lambdadef}
\lambda=\delta C^{-1},\qquad {\rm with}\qquad  \delta=\det(C)
\end{equation}
Specifically,  $\delta=r+1$, $4$, $3$, $2$, $1$ for $A_r$, $D_r$, $E_6$, $E_7$, $E_8$, 
respectively. These commutation relations correspond to \eqref{commuq},
with the matrix $\Lambda=\begin{pmatrix} 0 & \lambda \\
-\lambda & 0 \end{pmatrix}$.

We use  renormalized  cluster variables: 
$$\qQ_{\al,n}=\exp\left(-{1\over 2}\left({2i \pi \over \delta}+1\right)
\sum_\beta \lambda_{\al,\beta}\right)
\ \qX_{\al,n}\ ,$$
obeying the same commutation relations:
\begin{equation}\label{comq}
\qQ_{\al,n}\qQ_{\beta,m} = t^{\lambda_{\al,\beta}(m-n)} \qQ_{\beta,m}\qQ_{\al,n}, 
\quad m,n\in \{0,1\}.
\end{equation}

Then the mutations of the quantum cluster algebra are equivalent to the following recursion 
relations for the variables $\qQ_{\al,n}$:
\begin{equation}\label{qqsys}
t^{\lambda_{\al,\al}}\, \qQ_{\al,n+1}\qQ_{\al,n-1}
=\qQ_{\al,n}^2-\prod_{\beta\neq \al} \qQ_{\beta,n}^{-C_{\al,\beta}}
\quad (\al\in I_r;n\in \Z).
\end{equation}
Note that all the variables on the right hand side of \eqref{qqsys} commute with each other.
We call this equation the $\g$ quantum $Q$-system.
Again, the subset of renormalized cluster seeds which correspond to valid initial data for the quantum 
$Q$-system are parameterized by generalized Motzkin paths as in  
Equation \eqref{admissible}, and take the form: $\qY_\vm=(\qQ_{\al,m_\al};\qQ_{\al,m_\al+1})$.
Using the mutation, in the form of the quantum $Q$-system equation \eqref{qqsys}, one can prove
\begin{lemma}\label{comlem}
Within each valid initial data set $\qY_{\vm}$ with $\vm$ as in \eqref{admissible},
the solutions of the quantum $Q$-system \eqref{qqsys} have the following commutation relations:
\begin{equation}\label{commutq}
\qQ_{\al,n}\qQ_{\beta,m} =t^{\lambda_{\al,\beta}(m-n)} \qQ_{\beta,m}\qQ_{\al,n}
\quad (m,n\in \Z)
\end{equation}
\end{lemma}
\begin{proof}
By induction on $m-n$. The Lemma is true when $|m-n|\leq 1$ by \eqref{comq}. 
Assume that $\qQ_{\al,n}$ and $\qQ_{\beta,n+k}$ belong to the cluster seed $\qY_\vm$, 
where $k>1$, and that the Lemma holds for all $k'<k$. Note that if $k>1$, $\al\neq \beta$ 
if the two variables are in the same cluster. We use \eqref{qqsys}:
\begin{eqnarray*}
t^{\lambda_{\beta,\beta}}\qQ_{\al,n}\qQ_{\beta,n+k} &&=
\qQ_{\al,n}\left(\qQ_{\beta,n+k-1}^2 - \prod_{\gamma\neq\beta} 
\qQ_{\gamma,n+k-1}^{-C_{\gamma,\beta}}\right) \qQ_{\beta,n+k-2}^{-1} \\
&&\hskip-.5in= t^{-\lambda_{\al,\beta}(k-2)}\left(t^{2\lambda_{\al,\beta}(k-1)}\qQ_{\beta,n+k-1}^2-
t^{-(k-1)\sum_{\gamma\neq\beta}\lambda_{\al,\gamma}C_{\gamma,\beta}}\prod_{\gamma\neq\beta} 
\qQ_{\gamma,n+k-1}^{-C_{\gamma,\beta}}
\right)\qQ_{\beta,n+k-2}^{-1}\qQ_{\al,n}
\end{eqnarray*}
We use
$$
 2 \lambda_{\al,\beta}- \sum_{\gamma\neq \beta} \lambda_{\al,\gamma}C_{\gamma,\beta} = 
 \sum_{\gamma}\lambda_{\al,\gamma}C_{\gamma\beta} = \delta \delta_{\al,\beta}=0,
$$
where the last equality holds because $\al\neq \beta$. We conclude that
equation \eqref{commutq} holds for $k$.
\end{proof}

The cluster variables in any quantum cluster algebra have a Laurent property \cite{BZ} in an analogous 
way as for the usual cluster algebras. The solutions of the quantum $Q$-system inherit this property:
\begin{lemma}\label{lopol}
For any $\al,n$, $\qQ_{\al,n}$ can be expressed
as a (non-commutative) Laurent polynomial in any of the valid initial data sets $\qY_\vm$ with
 coefficients in $\Z[q,q^{-1}]$.
\end{lemma}
As in the commutative $Q$-system, we will use this property below to prove an analogue of the 
polynomiality property of solutions of the $Q$-system under special boundary conditions. 

In the following two sections, we will proceed in analogy to the commutative case as in \cite{DFK}, 
starting with the case of $A_1$ and then treating the case of simply-laced $\g$ for pedagogical reasons. 

\section{Graded tensor multiplicities and the quantum $Q$-system: the $A_1$ case}
Here, we derive a constant term formula for the graded $M$ and $N$
sums in Equations \eqref{msumintro} and \eqref{nsumintro} in terms of
solutions of the quantum $Q$-system. As in \cite{DFK}, we prove a slightly stronger statement,
where we change the $N$ and $M$ sums to be sums over a finite number of variables, $k<\infty$.
The equality of the sums in the introduction follows when $k$ is sufficiently large.

\subsection{The quantum $M=N$ formula}

Let $A$ be the $k\times k$ matrix with entries 
$A_{i,j}=\min(i,j)$.
For given $k$-tuples of non-negative integers $\bm=(m_1,m_2,...,m_k)^t$
and $\bn=(n_1,n_2,...,n_k)^t$ and some $\ell\in\Z$, we define
the following integers:
\begin{eqnarray*}
q_0&=& \ell +\sum_{i=1}^k i(2m_i-n_i),\quad
p_j=\sum_{i=1}^k \min(i,j)(n_i-2m_i)\quad (j=1,2,...,k)
\end{eqnarray*}
or equivalently $\bp=(p_1,p_2,...,p_k)^t=A(\bn-2\bm )$, and $q_0=\ell-p_k$.
We also define the quadratic form
\begin{equation}
Q(\bm,\bn)=-{1\over 2}\bm^t (\bp+A\bn) =\bm^t A(\bm-\bn)={1\over 4} (2\bm-\bn)^t A(2\bm-\bn)-
{1\over 4}\bn^t A\bn.
\end{equation}

For $a\in\Z_+$ and $b\in \Z$
the quantum binomial coefficient 
may be defined by using
 a generating function. Let $(x;q)_\infty=\prod_{i=1}^\infty (1-q^i x)$, then
\begin{equation} \label{defbin}
\sum_{a\in \Z_+} \left[\begin{matrix}a+b \\ a \end{matrix}\right]_q x^a ={(q^{b+1}x;q)_\infty\over (x;q)_\infty}
=\left\{
\begin{matrix}
\prod_{i=0}^{b} (1-q^i x)^{-1} & {\rm if} \, b\geq  0\\
\prod_{i=0}^{-b-1} (1-q^{i+b+1} x) & {\rm if} \, b < 0\\
\end{matrix}\right.
\end{equation}

Define
\begin{eqnarray}
M_{\ell,\bn}^{(k)}(q^{-1})&=&\sum_{m_1,m_2,...,m_k\in\Z_+\atop q_0=0;\,  p_1,p_2,...,p_k\geq 0} q^{Q(\bm,\bn)} 
\prod_{i=1}^k \left[\begin{matrix}m_i+p_i \\ m_i \end{matrix}\right]_q\label{msum} \\
N_{\ell,\bn}^{(k)}(q^{-1})&=&\sum_{m_1,m_2,...,m_k\in\Z_+\atop q_0=0} q^{Q(\bm,\bn)} 
\prod_{i=1}^k \left[\begin{matrix}m_i+p_i \\ m_i \end{matrix}\right]_q \label{nsum}
\end{eqnarray}
As we shall see below,
the sum $M_{\ell,\bn}^{(k)}(q)$ is identical to $M_{\lambda=\ell\omega_1,\bn}(q)$ of eq.\eqref{gramult} 
for $\g=\sl_2$, provided $k$ is large enough 
(and upon completing the sequence of $n_i$'s by zeros).
%

Note that the difference between the two sums is in the restriction of the sum  
\eqref{msum} to non-negative values of $p_i$. The unrestricted sum \eqref{nsum}, 
in contrast, includes terms which can be negative if $p_i<0$. We introduced the 
quantum $Q$-system in order to prove the graded version of the $M=N$ theorem 
from \cite{DFK}. In the $A_1$ case, we will show:
\begin{thm} \label{mequalsn} 
For any $\ell\geq0$ and $k\in\N$, $\bn\in \Z_+^k$,
$$
N_{\ell,\bn}^{(k)}(q) = M_{\ell,\bn}^{(k)}(q).
$$
\end{thm}
The proof is the subject of the following subsections. It consists of two steps: 
First, define a generating function in variables on the quantum torus, whose ``evaluated 
constant term", appropriately defined, is the $N$-sum. 
Then, observe that the evaluated constant term has no contributions from terms in the 
summation with values of $p_i$ which are strictly negative, for any $p_i$.


\subsection{Generating function for $N$ and $M$-sums}

Let $\bm,\bn\in \Z_+^k$
and choose $\ell\in\Z$. Recall the definition of  $q_0(\ell,\bn,\bm)$ in the previous section. We define the integers
\begin{equation*}
\bq(\bm,\bn)=\bp+q_0.
\end{equation*}
That is,
$$
 q_j(\bm,\bn)= \ell +\sum_{i=j+1}^k (i-j)(2m_i-n_i)\quad (j=1,2,...,k).
$$
Note that $q_k=\ell\geq 0$, and by definition, $q_j\big|_{q_0=0}=p_j.$ 
For fixed $k$, $q_{j+1}(\bm,\bn) = q_{j}(\bm',\bn')$, where $\bm_j'=\bm_{j+1}$ and $\bn_j'=\bn_{j+1}$. Similarly, 
\begin{equation}\label{qtrans}
q_{j+p}(\bm,\bn)=q_{j}(\bm^{(p)},\bn^{(p)})\hbox{ where }\bm^{(p)}_i=m_{i+p},\ \bn^{(p)}_i=n_{i+p}.\ 
\end{equation}
In these last expressions, we only have $k-p$ non-zero variables $m_j,n_j$. That is, $m_{k+j}=n_{k+j}=0$ if $j>0$.

One may rewrite the  function $Q(\bm,\bn)$ in terms of $\bq$:
\begin{lemma}\label{compuqQ}
In terms of the integers $\{q_i\}$ the function $Q(\bm,\bn)$ can be expressed as
\begin{eqnarray*}
Q(\bm,\bn)&=&{1\over 4}\sum_{j=0}^{k-1}\left( (q_j-q_{j+1})^2-(\sum_{i=j+1}^k n_i)^2\right)\\
\end{eqnarray*}
\end{lemma}
\begin{proof}
The following two properties of the integers $q_j$ are useful in what follows:
\begin{equation}\label{qrelation}
q_j-q_{j+1}=\sum_{i=j+1}^k (2m_i-n_i) \quad {\rm and}\quad q_{j-1}+q_{j+1}-2 q_j=2m_j-n_j, \ j\in [1,k]
\end{equation}
where $q_{k+1}=q_k$. Using the expression $Q(\bm,\bn) = \frac{1}{4} \left((2\bm - \bn)^t \bp - \bn^t A \bn\right)$, we have
\begin{eqnarray*}
-(2\bm-\bn)^t\bp&=&\sum_{i=1}^k (2m_i-n_i)(q_0-q_i)\\
&=&\sum_{i=1}^k((q_{i-1}-q_i)-(q_i-q_{i+1}))(q_0-q_i)\\
&=&
\sum_{i=1}^k (q_{i-1}-q_i)^2,
\end{eqnarray*}
by use of the Abel summation formula.
\end{proof}

We will need to use the $A_1$ quantum $Q$-system.
For $\g=\sl_2$ there is only one root: Let 
 $\qQ_j:=\qQ_{1,j}$. The quantum $Q$-system simplifies in this case:
\begin{equation}\label{qaonesys}
t \, \qQ_{j+1}\qQ_{j-1}=\qQ_{j}^2-1\qquad (j \in \Z)
\end{equation}
with commutation relations
\begin{equation}\label{qcomaone}
\qQ_j \qQ_{j+1}=t \,\qQ_{j+1}\qQ_j \qquad (j\in\Z).
\end{equation}
We choose
\begin{equation}\label{choiceq}q=t^{-2}
\end{equation}
and define the following generating function in the non-commutative variables $\qQ_0,\qQ_1$:
\begin{equation}\label{Zdef}
Z_{\ell;\bn}^{(k)}(\qQ_0,\qQ_1)=\sum_{m_1,m_2,...,m_k\in\Z_+} \qQ_1^{-q_0} \, \qQ_0^{q_1} 
q^{{1\over 4}(q_1^2+\sum_{i=1}^{k-1}(q_i-q_{i+1})^2)}
\prod_{i=1}^k 
\left[\begin{matrix}m_i+q_i \\ m_i \end{matrix}\right]_q  .
\end{equation}
Note that the power of $q$ which appears in the series is 
${\overline{Q}}(\bm,\bn)=Q(\bm,\bn)+
{1\over 4}\bn^t A\bn$ when $q_0=0$ in terms of the quadratic form
of Lemma \ref{compuqQ}.

Let $u=q^{\frac{1}{4}}$ and let $\C_u=\C[u,u^{-1}]$. 
Define $\mathcal R=\C_{u}[\qQ_0^{\pm1}]((\qQ_1^{-1}))$ to be the
ring of formal Laurent series in $\qQ_1^{-1}$ with coefficients in
$\C_{u}[\qQ_0^{\pm1}]$. Then $Z_{\ell;\bn}^{(k)}(\qQ_0,\qQ_1)\in
\mathcal R$. (Note that the power $q^{1/4}$ is only an artifact here: 
All the important functions will be expressed in terms of integer powers of $q$.)

To relate the generating series $Z_{\ell;\bn}^{(k)}(\qQ_0,\qQ_1)$ to
the $M$ and $N$ sums of Theorem \ref{mequalsn}, we define two maps, which
are non-commutative analogs of ``the constant term in $\qQ_1$" and
``the evaluation at $\qQ_0=1$" of functions in $\mathcal R$.  Any element of $\mathcal R$ can be written in the
normal-ordered form $f(\qQ_0,\qQ_1)=\sum_{a,b\in \Z} f_{a,b} \qQ_0^a
\qQ_1^b$, where $f_{a,b}\in \C_u$. We make the following definitions:

\begin{defn}\label{ct}
The constant term in $\qQ_1$ of $f$ is defined to be
\begin{equation}\label{CTf}
{\rm CT}_{\qQ_1}(f(\qQ_0,\qQ_1))=\sum_{a} f_{a,0} \qQ_0^a\in \C_u[\qQ_0^{\pm
  1}]. 
\end{equation}
Here, the sum over $a$ has a finite number of non-zero terms.
\end{defn}
\begin{defn}\label{eval}
The evaluation of $f$ at $\qQ_0=1$ is
\begin{equation}\label{evalf}
f(\qQ_0,\qQ_1)\vert_{\qQ_0=1}=\sum_{b\in \Z}  \qQ_1^b\sum_{a}f_{a,b}.
\end{equation}
Again, the sum over $a$ is finite.
\end{defn}

The two operations commute:
$$
\begin{CD}
\mathcal R @>{\rm CT}_{\qQ_1}>>\C_{u}[\qQ_0^{\pm1}] \\
@V{|_{\qQ_0=1}}VV @VV{|_{\qQ_0=1}}V\\
\C_{u}((\qQ_1^{-1}))@>{\rm CT}_{\qQ_1}>> \mathcal R
\end{CD}
$$
Therefore, we can compose them below without reference to order.

\begin{remark}
Our definition of the evaluation at $\qQ_0=1$ is really a
``left evaluation", as the $\qQ_0$'s have to be taken to the left of
all $\qQ_1$'s before evaluating. However, if we both take the constant
term in $\qQ_1$ and evaluate at $\qQ_0=1$, a ``right evaluation" would
yield the same result, as we have:
$$CT_{\qQ_1}\left(\sum_{a,b} f_{a,b}\qQ_0^a \qQ_1^b\right)
=CT_{\qQ_1}\left(\sum_{a,b} f_{a,b}t^{a b} \qQ_1^b\qQ_0^a
  \right)= \sum_{a} f_{a,0}\qQ_0^a$$ 
as the factor
$t^{ab}$ may be replaced by 1 because $b=0$ in
the constant term. As before, the sum over $a$ is finite and the
result is in $\C_{q^{1\over 4}}[\qQ_0^{\pm 1}]$.
\end{remark}

\begin{lemma}
The $N$ sum of \eqref{nsum} can be expressed as:
\begin{equation}\label{nsumZ}
N_{\ell,\bn}^{(k)}(q^{-1})=q^{-{1\over 4}\bn\cdot A\bn}\, \left.CT_{\qQ_1}\left(Z_{\ell;\bn}^{(k)}(\qQ_0,\qQ_1)\right)\right|_{\qQ_0=1}.
\end{equation}
\end{lemma}
\begin{proof}
The constant term ensures that $q_0=0$, therefore $q_i=p_i$ for the terms in the summation over $\bm$ which contribute to the constant term. Moreover, the quadratic forms ${\overline{ Q}}(\bm,\bn)-{1\over 4}\bn^tA\bn$ and
$Q(\bm,\bn)$ are identical when $q_0=0$, as remarked above.
\end{proof}

\subsection{The generating function in the case $k=1$}

We have the following identity involving $q$-binomial coefficients.
\begin{lemma}\label{powersum}
If we have two variables $x,y$ on the quantum torus, with  $y x=q \, x y$,  then
$$ \sum_{a\in \Z_+} \left[\begin{matrix}a+b \\ a \end{matrix}\right]_q x^a y^b=
y^{-1}\big(y(1-x)^{-1}\big)^{b+1} \qquad (b\in\Z), $$
where the right hand side is considered as a formal power series in the variable $x$.
\end{lemma}
\begin{proof}
Use the definition \eqref{defbin} and the commutation relations:
$y(1-x)^{-1} = (1-qx)^{-1} y$ for $b\geq 0$ and $y^{-1}(1-x)=(1-q^{-1} x)y^{-1}$ for $b<0$.
\end{proof}

Recall that $q=t^{-2}$, and therefore $\qQ_0\qQ_1^{-2}= q \qQ_1^{-2} \qQ_0$.
Applying Lemma \ref{powersum} with $y=\qQ_0$ and $x=\qQ_1^{-2}$, we get:
\begin{eqnarray*}
Z_{\ell;n}^{(1)}(\qQ_0,\qQ_1)&=&
q^{{1\over 4}\ell^2}\qQ_1^{n-\ell} \left(\sum_{m\geq 0} (\qQ_1^{-2})^m 
\left[\begin{matrix}m+\ell \\ m \end{matrix}\right]_q\right) \qQ_0^\ell\\
&=& q^{{1\over 4}\ell^2}\qQ_1^{n-\ell}  \qQ_0^{-1}\big( \qQ_0 (1-\qQ_1^{-2})^{-1}\big)^{\ell+1}\\
&=& q^{-{1\over 4}\ell}(\qQ_1^{-1})^{n+1} \qQ_0^{-1} \big( \qQ_0 \qQ_1^{-1}(1-\qQ_1^{-2})^{-1}\big)^{\ell+1}
\end{eqnarray*}

The quantum $Q$-system \eqref{qaonesys} allows to identify: 
$$\qQ_1\qQ_2^{-1}=t\qQ_1\qQ_0(\qQ_1^2-1)^{-1}= \qQ_0 \qQ_1^{-1}(1-\qQ_1^{-2})^{-1}\ ,$$
and
we may therefore rewrite the above result as:
\begin{lemma}
\begin{equation}\label{kone}
Z_{\ell;n}^{(1)}(\qQ_0,\qQ_1)=
q^{-{1\over 4}\ell}\qQ_1^{n+1} \qQ_0^{-1} \big( \qQ_1\qQ_2^{-1}\big)^{\ell+1}
\end{equation}
where $\qQ_2$
is the solution of the quantum $Q$-system \eqref{qaonesys} with initial data $(\qQ_0,\qQ_1)$.
\end{lemma}

\subsection{Factorization of the generating function}
We have defined the generating function $Z_{\ell,\bn}^{(k)}$ so that it can be summed, and has a factorization formula: This is the property which allows us to prove the $M=N$ identity of Theorem \ref{mequalsn}.

To show this, we first show a recursion relation for the generating function.
\begin{lemma}\label{telescope}
\begin{equation}\label{telsone}
Z_{\ell,\bn}^{(k)}(\qQ_0,\qQ_1)=
q^{-{1\over 2}n_1} Q_1Q_0^{-1} Q_1^{n_1+1} Q_2^{-1} Z_{\ell;\bn'}^{(k-1)}(\qQ_1,\qQ_2)
\end{equation}
\end{lemma}
\begin{proof}
From \eqref{qrelation},
\begin{equation}
\label{qrec} q_0 =2q_1-q_2 +2m_1-n_1 
\end{equation}
The summation over $m_1$ in the generating function \eqref{Zdef} can be performed explicitly, using the calcuation for the case $k=1$ leading to Lemma \ref{kone}, because none of the integers $q_i(\bm,\bn)$ in the binomial coefficients depend on it.
\begin{eqnarray*}
Z_{\ell;\bn}^{(k)}(\qQ_0,\qQ_1)&=&\sum_{\bm'\in\Z_+^{k-1}} 
q^{{1\over 4}(q_2^2+\sum_{i=2}^{k-1}(q_i-q_{i+1})^2)}
\prod_{i=2}^k \left[\begin{matrix}m_i+q_i \\ m_i \end{matrix}\right]_q  \times\\
&&\times\, 
q^{q_1(q_1-q_2)\over 2}
\qQ_1^{n_1+q_2-2q_1} \qQ_0^{-1}\big( \qQ_0(1-\qQ_1^{-2})^{-1}\big)^{q_1+1}.
\end{eqnarray*}
Using the quantum $Q$-system, the last term can be rewritten as
$$
q^{q_1(q_1-q_2)\over 2}
\qQ_1^{n_1+q_2-2q_1} \qQ_0^{-1}\big( \qQ_0(1-\qQ_1^{-2})^{-1}\big)^{q_1+1}=
q^{-{n_1\over 2}}\qQ_1\qQ_0^{-1}\qQ_1^{n_1+1}\, \qQ_2^{-q_1-1} \qQ_1^{q_2}.$$
One can now identify the summation over the remaining variables as the generating function with $k$ replaced by $k-1$, $\bm,\bn$ replaced by $\bm',\bn'$ (recall that that $q_{i+1}(\bm,\bn)=q_i(\bm',\bn')$) and the arguments $\qQ_0,\qQ_1$ are replaced by $\qQ_1,\qQ_2$. The Lemma follows.
\end{proof}

The quantum $Q$-system solutions have a translational invariance property as in the commutative case.
We make explicit reference to the initial conditions by denoting $\qQ_n(a,b)$ the solution
of the $Q$-system with initial data $\qQ_0=a$ and $\qQ_1=b$.

\begin{lemma} \label{translat}
The solution $\qQ_n(\qQ_0,\qQ_1)$ of the quantum $A_1$ $Q$-system \eqref{qaonesys}
satisfies the following translational invariance property:
$$ \qQ_n\big(\qQ_j,\qQ_{j+1}\big)=\qQ_{n+j}(\qQ_0,\qQ_1) 
 \qquad (n,j\in \Z_+).$$
\end{lemma}
\begin{proof}
The Lemma is true by definition for $n=2$, which is just the definition of the quantum $Q$-system. Suppose it is true for $m<n$. Then 
\begin{eqnarray*}
\qQ_n(\qQ_j,\qQ_{j+1})&=&\qQ_2(\qQ_{n-2}(\qQ_j,\qQ_{j+1}),\qQ_{n-1}(\qQ_j,\qQ_{j+1}))\\
& =& \qQ_2(\qQ_{n+j-2},\qQ_{n+j-1}) \\
&=& \qQ_{n+j}(\qQ_0,\qQ_1).
\end{eqnarray*}The Lemma follows by induction.
\end{proof}

Translational invariance together with the recursion \eqref{telsone} imply a complete factorization property:
\begin{thm}\label{factothm}
\begin{equation}\label{factotum}
Z_{\ell;\bn}^{(k)}(\qQ_0,\qQ_1)
=q^{-{1\over 2}\sum_{i=1}^k n_i-{1\over 4}\ell}\, 
\qQ_1\qQ_0^{-1}\left(\prod_{i=1}^k \qQ_i^{n_i}\right) (\qQ_k\qQ_{k+1}^{-1})^{\ell+1}
\end{equation}
where $\qQ_j$, $j\geq 0$ are the solutions of the quantum $A_1$ $Q$-system \eqref{qaonesys}
with initial data $(\qQ_0,\qQ_1)$.
\end{thm}
\begin{proof}
By induction on $k$. The Theorem holds for $k=1$ by explicit calcuation, Equation \eqref{kone}.
Suppose the theorem holds for $k-1$. Using Equation \eqref{telsone}, we have
\begin{eqnarray*}
Z_{\ell;\bn'}^{(k-1)}(\qQ_1,\qQ_2)&=&q^{-{1\over 2}\sum_{i=2}^k n_i-{1\over 4}\ell}
\left\{ \qQ_1\qQ_0^{-1}\left(\prod_{i=1}^{k-1} \qQ_i^{n_{i+1}}\right) 
(\qQ_{k-1}\qQ_{k}^{-1})^{\ell+1}\right\}_{\qQ_0\mapsto \qQ_1\atop \qQ_1\mapsto \qQ_2}\\
&=& q^{-{1\over 2}\sum_{i=2}^k n_i-{1\over 4}\ell}
\qQ_2\qQ_1^{-1}\left(\prod_{i=2}^k \qQ_{i}^{n_i}\right) (\qQ_k\qQ_{k+1}^{-1})^{\ell+1}
\end{eqnarray*}
where we have used Lemma \ref{translat} with $j=1$ in the second line
Substituting this into Equation \eqref{telsone},
and rearranging the factors by use of the commutation relations \eqref{qcomaone},  yields the theorem. 
\end{proof}

Theorem \ref{factothm} is useful as it allows us to write a factorization of the generating function in the following form:
\begin{cor}\label{factocor}
For any $j\in[1,k]$, we have:
\begin{equation}\label{factorbytwo}
Z_{\ell;n_1,...,n_k}^{(k)}(\qQ_0,\qQ_1)
=Z_{0;n_1,...,n_j}^{(j)}(\qQ_0,\qQ_1) \, Z_{\ell;n_{j+1},...,n_k}^{(k-j)}(\qQ_j,\qQ_{j+1}).
\end{equation}
\end{cor}
\begin{proof}
We rewrite the factorization in the form 
\begin{eqnarray*} 
q^{-{1\over 2}\sum_{i=1}^k n_i -{1\over 4}\ell}&&\!\!\!\!\!\!\!\!\!\!\!\!
\qQ_1\qQ_0^{-1}\Big(\prod_{i=1}^k \qQ_i^{n_i}\Big)(\qQ_k\qQ_{k+1}^{-1})^{\ell+1}
=\left(q^{-{1\over 2}\sum_{i=1}^j n_i}\qQ_1\qQ_0^{-1}\Big(\prod_{i=1}^j \qQ_i^{n_i} \Big)
\qQ_{j}\qQ_{j+1}^{-1}\right)\\
&\times&
\left(q^{-{1\over 2}\sum_{i=j+1}^k n_i -{1\over 4}\ell}
\qQ_{j+1}\qQ_j^{-1}\, \Big(\prod_{i=j+1}^k\qQ_i^{n_i}\Big)(\qQ_k\qQ_{k+1}^{-1})^{\ell+1}\right) 
\end{eqnarray*}
and use Lemma \ref{translat} to rewrite the second factor 
as $Z_{\ell;\bn^{(j)}}^{(k-j)}(\qQ_j,\qQ_{j+1})$.
\end{proof}

\subsection{Proof of the graded M=N identity}
To prove Theorem \ref{mequalsn}, we must show that there are no contributions to the constant term in the sum over $\bm$ in the generating function  \eqref{Zdef} from any $q_i<0$, when evaluated at $\qQ_0=1$. Contributions to the constant term have $q_i=p_i$, so this implies the identity of Theorem \ref{mequalsn}.

As in the commutative case, we use a ``descending induction" on $j=k,k-1,..,1$ to show that
the $N$-sum is unchanged if we restrict
the summations in the definition \eqref{nsum} to ${ q}_{j},...,{ q}_k\geq 0$.

The initial step is trivial, as $q_k=\ell\geq 0$.


We start from the expression \eqref{nsumZ} for the $N$-sum,
and use Corollary \ref{factocor} and Theorem \ref{factothm} to rewite it as:
\begin{eqnarray}
N_{\ell;\bn}^{(k)}(q^{-1})&=& q^{-{1\over 4}\bn\cdot A\bn-{1\over 2}\sum_{i=1}^k n_i-{1\over 4}\ell} CT_{\qQ_1}\left(
\qQ_1\qQ_0^{-1}\Big(\prod_{i=1}^j \qQ_i^{n_i} \Big)\qQ_{j}\qQ_{j+1}^{-1} \right. \nonumber \\
&\times &\left.
\sum_{m_{j+1},...,m_k\in\Z_+} \qQ_{j+1}^{-{q}_{j}} \qQ_j^{{q}_{j+1}} q^{{{q}_j^2\over 4}
+{1\over 4}\sum_{i=j+1}^k({q}_i-{q}_{i+1})^2}
\prod_{i=j+1}^k \left[\begin{matrix}m_i+{q}_i \\ m_i \end{matrix}\right]_q \right)\Big\vert_{\qQ_0=1} \label{factbarq}
\end{eqnarray}

The inductive step is as follows.
We assume the following recursion hypothesis holds:

$(H_j)$: {\it The $N$-sum is unchanged if
we restrict the summation over $\bm^{(j)}$ in the second factor to terms with ${q}_{j+1},...,{q}_{k}\geq 0$.}

We now wish to prove $(H_{j-1})$.
Consider the contribution to the sum \eqref{factbarq} from the terms with ${q}_j<0$. Recall that we may restrict the sum to terms with $q_{j+1}\geq 0$, by assumption. Each such term has strictly positive powers of $\qQ_j$, and the contribution to the 
factor in the pharentheses is proportional, up to a polynomial in $\C_q$,
to a monomial of the form:
$$ \qQ_1\qQ_0^{-1}\Pi_\nu(\qQ_0,\qQ_1), \qquad 
\Pi_\nu(\qQ_0,\qQ_1)=\prod_{i=1}^{j+1} \qQ_i^{\nu_i} $$
where all $\nu_i\in \Z_+$. 

We prove an analog of the polynomiality property for the commutative $Q$-system:

\begin{lemma}\label{vanilem}
For all $k\in \Z_{>0}$, and all $\nu_1,...,\nu_k\in \Z_+$ we have
$$CT_{\qQ_1}\left( \qQ_1\qQ_0^{-1} \Pi_\nu(\qQ_0,\qQ_1)\right)\Big\vert_{\qQ_0=1} =0 $$
\end{lemma}
\begin{proof}
By Lemma \ref{lopol}, each term $\qQ_i$ is expressible as a Laurent polynomial of 
the admissible initial data $(\qQ_{-1},\qQ_0)$. Therefore, 
there is an expression $\Pi_\nu(\qQ_0,\qQ_1)=\sum_{p\in \Z}  \qQ_{-1}^p c_p(\qQ_0)$, where
the Laurent polynomials $c_p(\qQ_0)$ are non-zero for finitely many $p$.
The constant term in $\qQ_1$ extracts from $\Pi_\nu$ the term with $p=1$, since
$\qQ_{-1}=t^{-1}\qQ_1^{-1}(\qQ_0^2-1)$. 
$$CT_{\qQ_1}\left( \qQ_1\qQ_0^{-1} \Pi_\nu(\qQ_0,\qQ_1)\right)
=\qQ_1\qQ_0^{-1}\qQ_{-1}c_1(\qQ_0)
=\qQ_0^{-1}(\qQ_0^2-1)c_1(\qQ_0)$$
which vanishes when evaluated at $\qQ_0=1$. The lemma follows.
\end{proof}

The Lemma implies that the $N$-sum \eqref{nsumZ}, expressed as \eqref{factbarq}, 
only receives non-vanishing contributions
from the sum over $m_i$'s such that $q_j={p}_j\geq 0$.  Hence the induction step ($H_{j-1}$) is proved.
This concludes the proof of Theorem \ref{mequalsn}.

\subsection{Computing the graded multiplicities}

In view of the $N=M$ identity, the constant term identity \eqref{nsumZ} may be rewritten as:
\begin{equation}\label{ctnsum}
M_{\ell;\bn}^{(k)}(q^{-1}) =q^{-{1\over 2}\sum_i n_i -{1\over 4}(\ell+\bn\cdot  A\bn)}
CT_{\qQ_1}\left( \qQ_1\qQ_0^{-1} \prod_{i=1}^k \qQ_i^{n_i} 
\left( \qQ_k \qQ_{k+1}^{-1}\right)^{\ell+1}
\right)\Bigg\vert_{\qQ_0=1},
\end{equation}
where $\qQ_i$ with $i>1$ are determined from the recursion \eqref{qaonesys} and satisfy the commutation relations \eqref{qcomaone}. The product over $i$ is taken with lower indices to the left of higher indices.

To extract the constant term, we have to express the $\qQ_i$ as
Laurent polynomials of $\qQ_0,\qQ_1$, while $z_k=\qQ_k \qQ_{k+1}^{-1}$ 
must be expanded as a formal
Laurent series of $\qQ_1^{-1}$, with coefficients Laurent polynomials in $\qQ_0$. 

\begin{lemma} The ratio $z_k$ is a power series in $\qQ_1^{-1}$.
\end{lemma}
\begin{proof}
From the recursion relation for $\qQ_i$ \eqref{qaonesys}, and the first two values $\qQ_0=\qQ_0$ and $\qQ_1=z_0^{-1}\qQ_0$,
it is easy to show that $\qQ_k$ is a Laurent polynomial in $\qQ_1$ of degree $k$,
with highest degree term of the form $z_0^{-k}\qQ_0=t^{k(k-1)/2} \qQ_0^{1-k}\qQ_1^k$. 
The quantum $Q$-system implies the recursion relation $z_{k+1}^{-1}z_k =1-\qQ_{k+1}^{-2}$. 
Considered in the ring of Laurent series in $\qQ_1^{-1}$, $\qQ_{k+1}^{-2}$ is a power series in $\qQ_1^{-1}$ with  
leading term in $\qQ_1^{-1}$ of the form $(z_0^{-k-1}\qQ_0)^{-2}$, hence of order $2k+2$.  Therefore,
$$
z_{k+1} = \qQ_0\qQ_1^{-1} \prod_{i=1}^{k+1} (1-\qQ_{i}^{-2})^{-1}
$$
is a power series in $\qQ_1^{-1}$ with no constant term.
\end{proof}
Moreover, writing $z_{k+1}-z_k=\qQ_{k+1}^{-2}z_{k+1}$, and recalling that $z_{k+1}$ has no constant term,
we see that the difference between $z_{k+1}$ and $z_k$ expanded as power series in $\qQ_1^{-1}$
occurs at degree $2k+3$. 

For a fixed list of parameters $\bn$, we  henceforth assume that $k$ is ``sufficiently large":
$$2k\geq \sum_i i\, n_i\ ,$$
the total degree in $\qQ_1$ of $\prod_i \qQ_i^{n_i}$.
Under this assumption,
the constant term of \eqref{ctnsum} is independent of $k$, and we can replace
$z_k$ by its stabilized limit:
\begin{equation}
\label{zdefz}
z =\lim_{k\to \infty} z_k = \qQ_0\qQ_1^{-1}\prod_{k=1}^\infty (1-\qQ_{k}^{-2})^{-1} 
\end{equation}
and we finally get:

\begin{thm}\label{compum}
With $z$ as in \eqref{zdefz}, and for $2k\geq \sum_i i\, n_i $, we have:
\begin{equation}\label{ctnsumfin}
M_{\ell;\bn}^{(k)}(q^{-1}) =M_{\ell;\bn}(q^{-1})=q^{-{1\over 2}\sum_i n_i -{1\over 4}(\ell+\bn\cdot  A\bn)}
CT_{\qQ_1}\left( \qQ_1\qQ_0^{-1} \prod_{i=1}^k \qQ_i^{n_i} z^{\ell+1}
\right)\Bigg\vert_{\qQ_0=1}
\end{equation}
where we may drop the superscript $(k)$ as the result doesn't depend on it. Taking $k\to \infty$ and
completing the vector $\bn$ with zero entries, we get $M_{\ell;\bn}(q^{-1})=M_{\ell\omega_1;\bn}(q^{-1})$,
the desired graded multiplicities.
\end{thm}

Let us describe how to concretely compute the  $M$'s by use of Theorem \ref{compum}.

\begin{defn}\label{unav}
Let $\Z_t= \Z[t,t^{-1}]$ and let $ p\in \Z_t[ \qQ_0^{-1},\qQ_0,\qQ_1,\qQ_2,...,\qQ_k]$ 
be a polynomial in the variables $\qQ_i (1\leq i \leq k) $ and a Laurent polynomial in $\qQ_0^{-1}$. To this polynomial
we associate the following element of the ground ring $\Z_t$:
$$ \mu_\ell({\p}) =
CT_{\qQ_1}\left( \qQ_1\qQ_0^{-1}\,p\,  z^{\ell+1}
\right)\Bigg\vert_{\qQ_0=1}. $$
We also define the moments $\mu_{\ell,j}$ by 
\begin{equation}\label{mulj}
\mu_{\ell,j}=\mu_\ell(Q_1^j).
\end{equation}
\end{defn}
We have the following properties:

\begin{lemma}\label{faclem}
For any polynomial of the form ${ p}=f(\qQ_0) g$,  where $f\in\Z_t[\qQ_0,\qQ_0^{-1}]$
and $g\in \Z_t[ \qQ_0^{-1},\qQ_0,\qQ_1,\qQ_2,...,\qQ_k]$,
$$ \mu_\ell({\p})= f(t^{-1}) \mu_\ell( g). $$
\end{lemma}
\begin{proof}
We use the definition:
\begin{eqnarray*} \mu_\ell({p})&=&CT_{\qQ_1}\left( \qQ_1\qQ_0^{-1} f(\qQ_0)g\,
z^{\ell+1}\right)\Bigg\vert_{\qQ_0=1}\\
&=&CT_{\qQ_1}\left( f(t^{-1}\qQ_0)  \qQ_1\qQ_0^{-1}g\,
z^{\ell+1}\right)\Bigg\vert_{\qQ_0=1}=f(t^{-1}) \mu_\ell(g),
\end{eqnarray*}
where the last expression is computed by (i) expressing $g$ as a Laurent polynomial of $(\qQ_0,\qQ_1)$,
(ii) moving all powers of $\qQ_0$'s to the {\it left} (iii) taking the constant term in $\qQ_1$
(iv) evaluating the expression at $\qQ_0=1$. 
\end{proof}

\begin{remark}\label{othereval}
We may restate the result of Lemma \ref{faclem} as follows: The evaluation 
at $\qQ_0=1$ of the function $p$ may be performed alternatively
by normal ordering the expression for $p$ first (putting all the $\qQ_0$'s in 
$p$ to the left) and setting $\qQ_0=t^{-1}$ in the normal ordered expression for $p$. 
This will be instrumental
in our algebraic reformulation below.
\end{remark}

\begin{lemma}{(Polynomiality Lemma)}\label{polpP}
For each polynomial ${p}\in\Z_t[ \qQ_0^{-1},\qQ_0,\qQ_1,\qQ_2,...,\qQ_k]$, there is a unique polynomial 
$P(\qQ_1)\in \Z_t[\qQ_1]$ such that:
$$ \mu_\ell(p)=\mu_\ell(P(\qQ_1)) .$$
\end{lemma}
\begin{proof}
By the Laurent polynomiality property of quantum cluster algebras of Lemma \ref{lopol}, 
we deduce that ${p}$ is a Laurent polynomial
of $(\qQ_0,\qQ_1)$ with coefficients in $\Z_t$. 

Suppose
$ p=\sum_{j\in \Z} \qQ_1^j f_j(\qQ_0) $. 
Since $z$ is a power series in $\qQ_1^{-1}$ with no constant term in $\qQ_1^{-1}$,
the constant term of the product $\qQ_1 \qQ_0^{-1} p\, z^{\ell+1}$  
vanishes for the terms in $p$ with negative values of $j$. So we are left with 
terms with non-negative values of $j$ in $ {p}$:
\begin{eqnarray*}
\mu_\ell({p})&=&\sum_{j\in\Z_+} \mu_\ell( \qQ_1^j f_j(\qQ_0) )\\
&=&\sum_{j\in\Z_+} \mu_\ell(f_j(t^{-j}\qQ_0) \qQ_1^j)=\sum_{j\in\Z_+}f_j(t^{-j-1}) \mu_\ell(\qQ_1^j)
\end{eqnarray*}
by use of Lemma \ref{faclem}. We conclude that $P(x)=\sum_{j\in\Z_+}f_j(t^{-j-1}) x^j$
and the lemma follows.
\end{proof}

Let us denote by $\varphi: \Z_t[\qQ_0^{\pm1},\qQ_1,...,\qQ_k]\to \Z_t[\qQ_1]$ 
the map taking the Laurent polynomial $ {p}$ to the polynomial $ P$ described in Lemma \ref{polpP}. 
We may summarize the action of $\varphi$ as ``left evaluation  at $Q_0=t^{-1}$ of the
part of $p$ with non-negative powers of $\qQ_1$", once $p$ is expressed as
a Laurent polynomial of $(\qQ_0,\qQ_1)$.
Moreover we have:  $\mu_\ell({p})=\mu_\ell(\varphi({p}))$.

To compute any graded multiplicity $M$, we only need information about the 
moments $\mu_{\ell,j}$ of Equation \eqref{mulj}, for each $j\in \Z_+$.
These in turn may be obtained either by computing $\mu_\ell(\qQ_m)$ as in Lemma \ref{polpP},
or from the definition of the corresponding $M$-sum for
$n_i=\delta_{i,m}$. This last is easily computed as $M_{\ell;\bn}(q^{-1})=t^{m+1}\mu_\ell(\qQ_m)=\delta_{m,\ell}$. 

Indeed, when $t=1$,
$M_{\ell;\bn}\equiv M_{\lambda=\ell\omega_1,\bn}$ is the
multiplicity of the irreducible representation $V_{\ell\omega_1}$ with highest weight $\ell\omega_1$
in the KR module $\KR_m=V_{m\omega_1}$ of $A_1$, itself equal to the irreducible
representation with highest weight $m\omega_1$, 
hence $\ell$ must be equal to $m$ and the multiplicity is $1$. Let us write
$$t^{-m-1}\delta_{\ell,m}= \mu_\ell(\qQ_m)=\mu_\ell(\varphi(\qQ_m))=\sum_{j=0}^m c_{m,j}(t)\mu_{\ell,j}$$
via Lemma \ref{polpP}. 
The coefficients $c_{m,j}(t)\in \Z[t,t^{-1}]$ are determined by the quantum $Q$-system, 
using the fact that 
\begin{equation}\label{basone}
\varphi(\qQ_m)=\sum_{0\leq j\leq m}  c_{m,j}(t) \qQ_1^j,\qquad c_{j,j}\neq 0,
\end{equation}
which is a triangular system and so may be inverted (see Remark \ref{chbas} below):
\begin{equation}\label{bastwo}
\qQ_1^j=\varphi(\sum_{0\leq i\leq j} d_{j,i}(t) \qQ_i )
\end{equation}
leading to 
the value of $\mu_{\ell,j}$:
$$\mu_{\ell,j}=\sum_{0\leq i\leq j}d_{j,i}(t) \mu_\ell(\qQ_i )=d_{j,\ell}(t)t^{-\ell-1}$$
Substituting these values in the expression of $P$ in Lemma
\ref{polpP} for ${p}=\prod_{i=1}^k \qQ_i^{n_i}$ finally yields:
$\mu_\ell({p})=\sum_{j\in\Z_+}f_j(t^{-j-1}) d_{j,\ell}(t)t^{-\ell-1}$.

\subsection{An algebraic reformulation of the fusion product}\label{vacvecsec}



In this section we use Remark \ref{othereval} to reformulate the quantity $\mu_\ell({\p} )$
as a special matrix element in a representation of the algebra generated by the non-commutative elements $\{\qQ_i\}$. 

Let ${\mathcal A}_+$ be the universal enveloping algebra over $\Z[t,t^{-1}]$
with the generators 
$$\qQ_0,\ \qQ_0^{-1},\ \qQ_m\quad (m\in \N).$$ These generators have relations determined by the quantum $Q$-system \eqref{qaonesys} 
and the commutation relations \eqref{qcomaone}. 

\begin{defn}\label{vacrep}
For $m\in \Z_+$, let $\langle m \vert$ denote the $\Z[t,t^{-1}]$--basis of a cyclic representation 
of the algebra ${\mathcal A}_+$,
such that:
\begin{eqnarray}
(i) \ \  \langle 0 \vert \qQ_0&=&t^{-1} \langle 0 \vert \label{zerocond} \\
(ii) \ \ \ \  \langle m \vert &=&t^{m+1}\langle 0 \vert \qQ_m \qquad (m>0) \label{basisvec}
\end{eqnarray}
\end{defn}

Equation \eqref{zerocond} gives an algebraic implementation of the left evaluation at $\qQ_0=t^{-1}$ 
described in Lemma \ref{faclem} and Remark \ref{othereval}.
We need the following preliminary lemma:

\begin{lemma}\label{vanione}
$$\langle 0 \vert \qQ_{-1} =0$$
\end{lemma}
\begin{proof}
Use the quantum $Q$-system to write 
$\qQ_{-1}=t^{-1}\qQ_1^{-1}(\qQ_0^2-1)=(t\qQ_0^2-t^{-1})\qQ_1^{-1}$,
and then use $\langle 0 \vert \qQ_0=t^{-1}\langle 0 \vert $ to conclude.
\end{proof}

We are now in a position to compute all the quantities 
$\langle 0 \vert \prod_i \qQ_i^{n_i}$, which are uniquely determined
by the above definitions, and the fact that the $\qQ$'s obey the quantum $Q$-system relations. 
To this end, let us formulate the following quantum counterpart of the polynomiality property
of cluster algebra that was used in the proof of the KR conjecture in the classical case \cite{DFK}.

\begin{lemma}\label{polyprop}
With the $\qQ_i$, $i\in \Z$, satisfying the quantum $Q$-system \eqref{qaonesys} 
and the commutation relations \eqref{qcomaone},
we have for an arbitrary polynomial ${\p}\in \mathcal A_+$:
$$\langle 0 \vert {\p}=\langle 0\vert \varphi({\p}), $$
where $P(\qQ_1)=\varphi(\p)$ is the map of Lemma  \ref{polpP}.
\end{lemma}
\begin{proof}
By Lemma \ref{lopol}, we may write 
${\p}$ as a Laurent polynomial of 
either initial data $(\qQ_0,\qQ_1)$
or $(\qQ_{-1},\qQ_0)$, and therefore we can write it as
$${\p}=\sum_{j\geq 0} \qQ_1^j f_j(\qQ_0)
 +\sum_{j>0} \qQ_{-1}^j g_{j}(\qQ_0)$$
where both $f_{j}$'s and $g_{j}$'s Laurent polynomials of $\qQ_0$ with coefficients in $\Z[t,t^{-1}]$. 
Using Lemma \ref{vanione}, 
\begin{eqnarray*}
\langle 0 \vert {\p}&=& \sum_{j\geq 0} \langle 0 \vert \qQ_1^j f_{j}(\qQ_0)\\
&=& \langle 0 \vert \sum_{j\geq 0}  f_{j}(t^{-j-1})\qQ_1^j =\langle 0 \vert \varphi({\p})(\qQ_1)
\end{eqnarray*}
\end{proof}

Let $\vert \ell \rangle$, $\ell\in \Z_+$, denote the dual vector basis to that of Def. \ref{vacrep}, 
namely such that 
$$\langle m \vert \ell \rangle=\delta_{m,\ell}\qquad {\rm for} \quad  \ell,m\in\Z_+$$
We deduce the main theorem:

\begin{thm}
With the $\qQ_i$, $i\in \Z$, satisfying the quantum $Q$-system \eqref{qaonesys} 
and the commutation relations \eqref{qcomaone},
we have for an arbitrary polynomial ${\p}\in\mathcal A_+$,
$$ \mu_\ell({\p})= \langle 0 \vert {\p} \vert \ell \rangle $$
\end{thm}
\begin{proof}
By Lemmas \ref{polpP} and  \ref{polyprop}, we only have to check that 
$\langle 0 \vert\qQ_1^j \vert \ell \rangle=\mu_{\ell,j}$ for all $\ell,j\in \Z_+$.
Using the ``change of basis" \eqref{bastwo}, this reduces to 
$\langle 0 \vert\qQ_m  \vert \ell \rangle= \mu_\ell(\qQ_m)$, which follows directly from Definition
\ref{vacrep}.
\end{proof}

\begin{remark}\label{chbas}
The change of basis from
$\langle m \vert$, $m\in \Z_+$ to $\langle 0 \vert \qQ_1^{j}$, $j \in \Z_+$
can be made completely explicit, by using the following linear recursion relation, 
satisfied by 
the solution of the quantum $Q$-system \eqref{qaonesys}:
\begin{equation}\label{linrecone}
\qQ_{n+1}+t \qQ_{n-1}=\big( \qQ_1(\qQ_0)^{-1}+t \qQ_{-1}(\qQ_0)^{-1}\big) \qQ_n
\qquad (n\in \Z)
\end{equation}
Using this relation and induction, the change of basis reads:
For all $m,n\in\Z_+$:
\begin{eqnarray}
\langle m\vert &=& \sum_{j=0}^{[m/2]} (-1)^j\,  t^{{m(m+3)\over 2}-j(j+1)}
\left[\begin{matrix}m-j\\ m-2j\end{matrix}\right]_q\,  \langle 0\vert (\qQ_1)^{m-2j}\\
\langle 0\vert (\qQ_1)^n &=&   \sum_{j=0}^{[n/2]} t^{-{(n - 2 j) (n - 2 j + 3)\over 2}}
\left\{\left[\begin{matrix}n\\ j\end{matrix}\right]_q
-\left[\begin{matrix}n\\ j-1\end{matrix}\right]_q\right\} \langle n-2j\vert 
\end{eqnarray}
\end{remark}



The above construction is summarized in the following reformulation of $M_{\ell;\bn}$:

\begin{thm} \label{oscirepone}
We have:
$$M_{\ell;\bn}(q^{-1}) = q^{-{1\over 2}\sum_i n_i -{1\over 4}(\ell+\bn\cdot  A\bn)} \langle 0 \vert
\prod_{i=1}^k \qQ_i^{n_i} \vert \ell \rangle $$
Up to prefactors, the fusion product of the KR modules $V_{i\omega_1}$ $n_i$ times, $i=1,2,...,k$
is realized on the representation of Definition \ref{vacrep} by the action of the product of the corresponding
quantum $Q$-system solutions, $\qQ_i^{n_i}$, $i=1,2,...,k$.
\end{thm}

\begin{example}
Let us compute the $n$-fold fusion product of $V_{\omega_1}$, $(V_{\omega_1})^{\star n}$.
For integers $n\geq \ell\geq 0$ with $\ell=n$ mod 2, we have by Remark \ref{chbas}:
$$\langle 0\vert (Q_1)^n \vert \ell\rangle = t^{-{\ell (\ell+ 3)\over 2}}\left\{\left[\begin{matrix}n\\ {n-\ell\over 2}\end{matrix}\right]_q
-\left[\begin{matrix}n\\{n-\ell\over 2}-1\end{matrix}\right]_q\right\}$$
We deduce:
$$M_{\ell;(n,0,0,...)}(q^{-1})=t^{n+{\ell\over 2}+{n^2\over 2}}\langle 0\vert (Q_1)^n \vert \ell\rangle 
= q^{-{(n-\ell) (n+\ell+ 2)\over 4}}\left\{\left[\begin{matrix}n\\ {n-\ell\over 2}\end{matrix}\right]_q
-\left[\begin{matrix}n\\{n-\ell\over 2}-1\end{matrix}\right]_q\right\}$$
This number was computed in \cite{Ke04} in terms of the co-charge $q$-Kostka polynomials:
$$M_{\ell;(n,0,0,...)}(q^{-1})={\tilde K}_{(n-j,j),(1^n)}(q^{-1})=q^{-{n(n-1)\over 2}} K_{(n-j,j),(1^n)}(q) $$
with $j=(n-\ell)/2$, and $K_{\lambda,\mu}(q)$ the $q$-Kostka polynomials.
The above formula agrees with the Hook formula for the $q$-Kostka polynomials \cite{Ke04}:
$$q^{-{n(n-1)\over 2}}K_{(n-j,j),(1^n)}(q) =q^{n(2^j1^{n-2j})-{n(n-1)\over 2}} 
{\prod_{i=1}^n (1-q^{i})\over \prod_{x\in (n-j,j)} (1-q^{h(x)})}$$
where $n(\lambda)=\sum(i-1)\lambda_i$ and $h(x)$ is the hook length of the box $x$ namely
the total number of boxes to the right of $x$ plus those below $x$, plus one. This
gives $n(2^j1^{n-2j})=j(j-1)+n$, so that
$$q^{-{n(n-1)\over 2}}K_{(n-j,j),(1^n)}(q)=q^{-j(n-j)} 
{\prod_{i=1}^n (1-q^{i})\over \prod_{x\in (n-j,j)} (1-q^{h(x)})}$$
This is identical to:
$$ q^{-j(n-j+1)}\left\{\left[\begin{matrix}n\\ j\end{matrix}\right]_q
-\left[\begin{matrix}n\\j-1\end{matrix}\right]_q \right\}=q^{-j(n-j)}\, {\prod_{i=1}^n (1-q^i) \over
\prod_{i=1}^j (1-q^i) \prod_{m=1}^{n-2j} (1-q^m) \prod_{k=n-2j+1}^{n-j}(1-q^{k+1})} $$
equal to $q^{-j(n-j)}$ times the $q$-hook product for the Young diagram with two rows, one of length
$n-j$ and one of length $j$. We conclude that:
$$(V_{\omega_1})^{\star\, n}={\begin{matrix}\oplus \\ {}_{ 0\leq  2j\leq n} \end{matrix}}
{q^{-j(n-j)}\prod_{i=1}^n (1-q^i) \over
\prod_{i=1}^j (1-q^i) \prod_{m=1}^{n-2j} (1-q^m) 
\prod_{k=n-2j+1}^{n-j}(1-q^{k+1})}\, V_{(n-2j)\omega_1}$$
\end{example}

\begin{example}
Let us compute the fusion product $V_{2\omega_1}\star V_{2\omega_1}$. We first express:
$$\langle 0\vert \qQ_2^2=t^{-2}\langle 0\vert (\qQ_1^2-1)\qQ_0^{-1}(\qQ_1^2-1)\qQ_0^{-1}
=\langle 0\vert(t^6\qQ_1^4-(t^2+t^4)\qQ_1^2+1)$$
by commuting all the $\qQ_0$'s to the left.
We then write the result in the $\langle m\vert$ basis, using Lemma \ref{chbas}:
$$\langle 0\vert \qQ_2^2=t^{-8} \langle 4 \vert +t^{-5} \langle 2\vert +t^{-2} \langle 0\vert $$
Collecting the prefactors, we finally get
$$ M_{4;0,2}=1 \qquad M_{2;0,2}=t^2 \qquad M_{0;0,2}=t^4$$
and therefore 
$$V_{2\omega_1}\star V_{2\omega_1}=V_{4\omega_1}\oplus q^{-1} V_{2\omega_1}\oplus q^{-2}V_0$$
\end{example}

\section{Graded tensor multiplicities and the quantum $Q$-system: The case of simply-laced algebras}

We shall repeat the analysis of the $M$ and $N$-sums for arbitrary simply-laced Lie algebras,
with rank $r$ and Cartan matrix $C=(C_{\al,\beta})_{\al,\beta\in  I_r}$. Formally, the arguments are almost 
identical to the case of $A_1$, but we now have to keep track of more indices. For this reason, we will 
sometimes use the shorthand $\bqQ_k$ to mean the set of (commuting) variables $\{\qQ_{1,k},...,\qQ_{r,k}\}$, 
and so on, where $r$ is the rank of the algebra.

\subsection{$M$ and $N$ sums}

Given $(r\times k)$-tuples of integrers $\bm=(m_{\al,i})_{\al\in I_r;i\in[1,k]}$,
$\bn=(n_{\al,i})_{\al\in I_r;i\in[1,k]}$, and integer $r$-tuples $\bell=(\ell_{\al})_{\al\in  I_r}$, 
we define the elements $\bp=(p_{\al,j})_{\al\in I_r,j\in 1,...,k}$ and $\bq_0=(q_{\al,0})_{\al\in I_r}$ as
\begin{eqnarray*}
\bp&=&(I\otimes A) \bn -(C\otimes A) \bm,\\
\bq_{0}&=&\bell-\bp_k,
\end{eqnarray*}
where the $k\times k$ matrix $A$ has entries $A_{i,j}=\min(i,j)$.
We also
introduce the quadratic form:
\begin{eqnarray*} Q(\bm,\bn)&=&-{1\over 2} \bm\cdot (\bp+(I\otimes A)\bn)\\
&=&-{1\over 2}(\bm-(C^{-1}\otimes I) \bn)\cdot \bp - {1\over 2}(C^{-1}\otimes I) 
\bn\cdot (I\otimes A)\bn\\
&=&
{1\over 2}(\bm-(C^{-1}\otimes I) \bn)\cdot (C\otimes A)(\bm-(C^{-1}\otimes I) \bn)- {1\over 2}(C^{-1}\otimes I) 
\bn\cdot (I\otimes A)\bn 
\end{eqnarray*}

and the sums:

\begin{eqnarray}
M_{\bl;\bn}^{(k)}(q^{-1})&=&\sum_{m_{i,\al}\in \Z_+\atop
q_{\al,0}=0; \, p_{\al,i}\geq 0}q^{Q(\bm,\bn)}\prod_{\al=1}^r 
\prod_{i=1}^k  \left[\begin{matrix}m_{\al,i}+p_{\al,i} \\ m_{\al,i} \end{matrix}\right]_q\label{Msumsl}\\
N_{\bl;\bn}^{(k)}(q^{-1})&=&\sum_{m_{i,\al}\in \Z_+\atop
q_{\al,0}=0} q^{Q(\bm,\bn)}\prod_{\al=1}^r
\prod_{i=1}^k  \left[\begin{matrix}m_{\al,i}+p_{\al,i} \\ m_{\al,i} \end{matrix}\right]_q\label{Nsumsl}
\end{eqnarray}
Note that the integers $\bq_0$ were introduced for the purpose of imposing the restriction on the 
summation variables (also known as the zero weight condition, see \cite{AK}).

The next sections are devoted to the proof of the following
\begin{thm}\label{M=Nsl}
$$M_{\bl;\bn}^{(k)}(q)=N_{\bl;\bn}^{(k)}(q)$$
for all $\bl \in \Z_+^r$ and all $\bn\in \Z_+^r\times \Z_+^k$.
\end{thm}

\subsection{Generating functions}

We define the integers
$$ q_{\al,j}=q_{\al,0}+p_{\al,j}=\ell_\al+\sum_{i=j+1}^k\sum_{\beta=1}^r 
(i-j)(C_{\al,\beta} m_{\beta,i} -\delta_{\al,\beta} n_{\beta,i})\quad (\al\in I_r;j\in[1,k])$$
which have the property that $\bq|_{\bq_0=0}=\bp$.
As above, we denote by $\bq_i=(q_{\al,i})_{\al\in  I_r}\in\Z^r$, and $\bn_i=(n_{\al,i})_{\al\in  I_r}\in\Z_+^r$.

\begin{lemma}\label{valQ}
$$Q(\bm,\bn)={1\over 2 \delta}\sum_{j=1}^k \left((\bq_{j-1}-\bq_{j})\cdot \lambda(\bq_{j-1}-\bq_{j})-
(\sum_{i=j+1}^k \bn_{i})\cdot \lambda(\sum_{i=j+1}^k \bn_{i})\right) $$
where the matrix $\lambda$ and the number $\delta$ are defined in \eqref{lambdadef}.
\end{lemma}
\begin{proof}
We note first that 
\begin{eqnarray}
q_{\al,j}-q_{\al,j+1}&=&\sum_{i=j+1}^k \sum_{\beta=1}^r (C_{\al,\beta}m_{\beta,i}-\delta_{\al,\beta}n_{\al,i}) 
\label{firstdif} \\
q_{\al,j-1}+q_{\al,j+1}-2 q_{\al,j}&=&\sum_{\beta=1}^r C_{\al,\beta}m_{\beta,j}-n_{\al,j}\label{secdif}
\end{eqnarray}
valid for $j=1,2,...,k$ provided we define $q_{\al,k+1}=q_{\al,k}$ for all $\al\in I_r$.
Recalling that $C^{-1}={1\over \delta}\lambda$, we compute
\begin{eqnarray*}
-{1\over 2}(\bm-(C^{-1}\otimes I) \bn)\cdot \bp&=&{1\over 2\delta}
\sum_{i=1}^k\sum_{\al,\beta=1}^r 
(\sum_{\gamma=1}^rC_{\al,\gamma}m_{\gamma,i}-n_{\al,i})\lambda_{\al,\beta}(q_{\beta,0}-q_{\beta,i})\\
&=&{1\over 2\delta}\sum_{i=1}^k\sum_{\al,\beta=1}^r ((q_{\al,i-1}-q_{\al,i})-(q_{\al,i}-q_{\al,i+1}))
\lambda_{\al,\beta}(q_{\beta,0}-q_{\beta,i})\\
&=&
{1\over 2 \delta}\sum_{i=1}^k\sum_{\al,\beta=1}^r  (q_{\al,i-1}-q_{\al,i})\lambda_{\al,\beta}(q_{\beta,i-1}-q_{\beta,i})
\end{eqnarray*}
by use of the Abel summation formula. 
The remaining term ${1\over 2}(C^{-1}\otimes I)\bn\cdot (I\otimes A)\bn ={1\over 2\delta}\bn^t (\lambda\otimes A)\bn$
is obtained by formally setting $m_{\al,i}=0$ for all $\al,i$ in the above, and the Lemma follows.
\end{proof}

Let us fix the quantum parameter to be
\begin{equation}\label{cutie} q=t^{-\delta}\end{equation}
with $\delta$ as in \eqref{lambdadef}.
For any ring $R$ and a set of variables $x=\{x_1,...,x_n\}$,
let $R((x))$ denote the ring of formal Laurent series of the variables $x_1,...,x_n$.
As above, we define the generating series for multiplicities
$$Z_{\bl;\bn}^{(k)}(\bqQ_0,\bqQ_1
)\in 
\Z_{q^{1\over 2\delta}}[\{\qQ_{\al,0}^{\pm 1}\}_{\al\in I_r}]((\{\qQ_{\al,1}^{-1}\}_{\al\in I_r}))$$
with the non-commutative variables $\{\qQ_{\al,0},\qQ_{\al,1}, \al \in I_r\}$ subject
to the commutation relations \eqref{comq}.
The generating functions are defined as follows:
\begin{eqnarray}
\label{Zsl}\\
Z_{\bl;\bn}^{(k)}(\bqQ_0,\bqQ_1
)=\sum_{\underset{\al\in I_r;\, i\in [1,k]}{m_{\al,i}\in\Z_+} } 
q^{\overline{Q}(\bm,\bn)}
\prod_{\al=1}^r (\qQ_{\al,1})^{-q_{\al,0}} 
\prod_{\al=1}^r (\qQ_{\al,0})^{q_{\al,1}} 
\prod_{\al\in I_r\atop i=1,...,k}  \left[\begin{matrix}m_{\al,i}+q_{\al,i} \\ m_{\al,i} \end{matrix}\right]_q
\nonumber
\end{eqnarray}
Here, the modified quadratic function
$${{\overline Q}(\bm,\bn)}={1\over 2\delta}\big(\bq_1\cdot \lambda \bq_1+\sum_{i=1}^{k-1}(\bq_i-\bq_{i+1})
\cdot\lambda(\bq_i-\bq_{i+1})\big)$$ has the property that it is equal to
$Q(\bm,\bn)+{1\over 2\delta}\bn^t  (\lambda\otimes A)\bn$ when $\bq_0=0$,
with $Q(\bm,\bn)$ as in Lemma \ref{valQ}. 


The generating function $Z_{\bl;\bn}^{(k)}(\bqQ_0,\bqQ_1
)$ 
is  related to the $N$-sum 
\eqref{Nsumsl} via a constant term and an evaluation. 
For a Laurent series $f(\bqQ_0,\bqQ_1
)$
in the variables $\{\qQ_{\al,1}^{-1}\}_{\al\in I_r}$,
with coefficients which are Laurent polynomials of the variables $\{\qQ_{\al,0}\}_{\al\in I_r}$, 
we define the multiple-constant term 
$\CT_{\bqQ_{1}}(f)$ to be the term of total degree $0$ in each of the variables
$(\qQ_{1,1}$, $\qQ_{2,1}$,..., $\qQ_{r,1})$ 
in any formal expansion of $f$.
In particular, if we have a normal-ordered  expansion 
$$f=\sum_{a_1,...,a_r,b_1,...,b_r\in \Z} f_{a_1,...,a_r;b_1,...,b_r}
\prod_{\al=1}^r\qQ_{\al,0}^{a_\al}\, \prod_{\beta=1}^r\qQ_{\beta,1}^{b_\beta},
$$
we have
$$ \CT_{\bqQ_1}(f)=\sum_{a_1,...,a_r}f_{a_1,...,a_r;0,...,0}\prod_{\al=1}^r\qQ_{\al,0}^{a_\al}. $$
Note that the sum is finite.
Likewise, 
we define the multiple evaluation of $f$ at $\qQ_{1,0}=1,...,\qQ_{r,0}=1$ to be the Laurent series:
$$f\big|_{\bqQ_0=1}=\sum_{a_1,...,a_r,b_1,...,b_r} 
f_{a_1,...,a_r;b_1,...,b_r}\prod_{\beta=1}^r\qQ_{\beta,1}^{b_\beta}$$
As in the $A_1$ case, this is a ``left evaluation". The constant term and evaluation maps commute, and their
composition gives:
$$\CT_{\bqQ_1}(f)\big|_{\bqQ_0=1}=\sum_{a_1,...,a_r} f_{a_1,...,a_r;0,...,0} .$$
The same result would be obtained with a ``right evaluation" because all the variables $t$-commute.

We may now express the $N$-sum in terms of $Z_{\bl;\bn}^{(k)}(\bqQ_0,\bqQ_1)$ as:
\begin{equation}
\label{NsumSL} N_{\bl;\bn}^{(k)}(q^{-1})=q^{- {1\over 2\delta}
\bn^t(\lambda\otimes A)\bn } \, 
\CT_{\bqQ_1}\Big( Z_{\bl;\bn}^{(k)}(\bqQ_0,\bqQ_1) 
\Big)\Big\vert_{\bqQ_0=1} 
\end{equation}
where the constant term ensures the condition $\bq_0=0$, and the result agrees with the definition 
\eqref{Nsumsl}, as ${\overline Q}(\bm,\bn)-{1\over 2\delta}
\bn^t (\lambda\otimes A)\bn$ and $Q(\bm,\bn)$ are identical when $\bq_0=0$.

\subsection{Factorization: the case $k=1$}

We first compute $Z_{\bl;\bn}^{(1)}(\bqQ_0,\bqQ_1)$, by explicitly summing
over the variables $\{m_{\al,1}, \al\in  I_r\}$. In this case,
$q_{\al,0}=\ell_\al+\sum_\beta C_{\al,\beta}m_{\beta,1} -n_{\al,1}$ and $q_{\al,1}=\ell_\al$,
so that:
\begin{equation*}
Z_{\bl;\bn}^{(1)}(\bqQ_0,\bqQ_1)=q^{{1\over 2\delta}\bl\cdot \lambda \bl}
\prod_{\al=1}^r \qQ_{\al,1}^{n_{\al,1}-\ell_\al} 
\left(\prod_{\beta=1}^r 
\sum_{m_{\beta,1}\in \Z_+} \big(\prod_{\al=1}^r \qQ_{\al,1}^{-C_{\al,\beta}}\big)^{m_{\beta,1}}
\left[\begin{matrix}m_{\beta,1}+\ell_\beta \\ m_{\beta,1} \end{matrix}\right]_q \right)
\prod_{\gamma=1}^r \qQ_{\gamma,0}^{\ell_\gamma}
\end{equation*}
Using the commutation relations \eqref{comq}, we get:
$$ \qQ_{\gamma,0}\left(\prod_{\al=1}^r \qQ_{\al,1}^{-C_{\al,\beta}}\right)
=t^{-\sum_{\al=1}^r\lambda_{\gamma,\al}C_{\al,\beta}}\,
\left(\prod_{\al=1}^r \qQ_{\al,1}^{-C_{\al,\beta}}\right) \, \qQ_{\gamma,0}=
q^{\delta_{\gamma,\beta}}
\,\left(\prod_{\al=1}^r \qQ_{\al,1}^{-C_{\al,\beta}}\right) \, \qQ_{\gamma,0}$$
by use of $\lambda C=\delta I$ and \eqref{cutie}.
We may consequently apply Lemma \ref{powersum} to each summation, resulting in:
\begin{eqnarray*}
&&\!\!\!\!\!\!\!\! Z_{\bl;\bn}^{(1)}(\bqQ_0,\bqQ_1)=
q^{{1\over 2\delta}\bl\cdot \lambda \bl}
\prod_{\al=1}^r \qQ_{\al,1}^{n_{\al,1}-\ell_\al}  \prod_{\beta=1}^r \qQ_{\beta,0}^{-1} 
\prod_{\gamma=1}^r
\Big(\qQ_{\gamma,0}(1-\prod_{\al=1}^r \qQ_{\al,1}^{-C_{\al,\gamma}})^{-1}\Big)^{\ell_\gamma+1}\\
&=&q^{-{1\over 2 \delta}\sum_{\al=1}^r \lambda_{\al,\al}\ell_\al}
q^{{1\over \delta}\sum_{\al<\beta}\lambda_{\al,\beta}}
\prod_{\al=1}^r \qQ_{\al,1}^{n_{\al,1}+1}
 \prod_{\beta=1}^r \qQ_{\beta,0}^{-1}\prod_{\gamma=1}^r
\Big(\qQ_{\gamma,0} \qQ_{\gamma,1}^{-1} (1-\prod_{\al=1}^r \qQ_{\al,1}^{-C_{\al,\gamma}})^{-1}\Big)^{\ell_\gamma+1}
\end{eqnarray*}
This is easily rewritten as follows, in terms of $\qQ_{\beta,2}=t^{-\lambda_{\beta,\beta}} 
(\qQ_{\beta,1}^{2}-\prod_{\al\neq \beta}\qQ_{\al,1}^{-C_{\al,\beta}})\qQ_{\beta,0}^{-1}$,
the solution of the quantum $Q$-system \eqref{qqsys} with initial
data $(\{\qQ_{\al,0}\},\{\qQ_{\al,1}\})$:
\begin{lemma}\label{zonesl}
\begin{equation}
Z_{\bl;\bn}^{(1)}(\bqQ_0,\bqQ_1)=
q^{-{1\over 2 \delta}\sum_{\al=1}^r \lambda_{\al,\al}\ell_\al}
\, \prod_{\al=1}^r \qQ_{\al,1}^{n_{\al,1}}
\prod_{\beta=1}^r \qQ_{\beta,1}\qQ_{\beta,0}^{-1} \prod_{\gamma=1}^r
\big(\qQ_{\gamma,1}\qQ_{\gamma,2}^{-1}\big)^{\ell_\gamma+1}
\end{equation}
\end{lemma}
Note that the products over non-commuting factors in this formula are taken to be ordered with lower 
indices to the left and higher indices to the right.

\subsection{Factorization: the general $k$ case}
We first prove a recursion relation:
\begin{lemma}
\begin{eqnarray}
&&Z_{\bl;\bn}^{(k)}(\bqQ_0,\bqQ_1) \nonumber \\
&&\qquad =q^{-{1\over \delta}\sum_{\al,\beta}n_{\al,1}\lambda_{\al,\beta}}
\prod_{\al=1}^r \qQ_{\al,1}
\prod_{\beta=1}^r \qQ_{\beta,0}^{-1}\prod_{\al=1}^r \qQ_{\al,1}^{n_{\al,1}+1} 
\prod_{\beta=1}^r \qQ_{\beta,2}^{-1}\,
Z_{\bl;\bn'}^{(k-1)}(\bqQ_1,\bqQ_2)\label{recuZsl}
\end{eqnarray}
where $\qQ_{\al,2}$ is the solution of the quantum $Q$-system \eqref{qqsys}
with initial data $(\{\qQ_{\al,0}\},\{\qQ_{\al,1}\})$. 
\end{lemma}
We use the same notation as in the $A_1$ case, with $\bn' = (n_{\al,2}, ...,n_{\al,k})_{\al\in I_r}$ and so forth.

\begin{proof}
From Equation \eqref{secdif} with $j=1$, 
$$q_{\al,0}=2 q_{\al,1}-q_{\al,2} +\sum_{\beta=1}^rC_{\al,\beta}m_{\beta,1}-n_{\al,1}.
$$
One can explicitly perform the summations over $m_{\al,1}\in \Z_+$ for all $\al$,
using Lemma \ref{powersum}. Then \eqref{Zsl} can be re-written as:
\begin{eqnarray*}
&Z_{\bl;\bn}^{(k)}(\bqQ_0,\bqQ_1)&=
\sum_{m_{\al,2},...,m_{\al,k}\in \Z_+\atop \al\in  I_r} 
q^{{1\over 2\delta}\big(\bq_2\cdot \lambda \bq_2+\sum_{i=2}^{k-1} 
(\bq_i-\bq_{i+1})\cdot\lambda(\bq_i-\bq_{i+1})\big)} \prod_{i=2}^k
\left[\begin{matrix}m_{\al,i}+q_{\al,i} \\ m_{\al,i} \end{matrix}\right]_q 
\nonumber \\
&\times \ q^{{1\over \delta}\bq_1\cdot \lambda(\bq_1-\bq_2)}&
\prod_{\al=1}^r \qQ_{\al,1}^{q_{\al,2}+n_{\al,1}-2 q_{\al,1}}
\prod_{\beta=1}^r \qQ_{\beta,0}^{-1}\Big(\qQ_{\beta,0} (1-\prod_{\al=1}^r 
\qQ_{\al,1}^{-C_{\al,\beta}})^{-1}\Big)^{q_{\beta,1}+1}
\end{eqnarray*}
We use the commutation relations \eqref{comq} and the quantum $Q$-system to rewrite the last factor as
\begin{eqnarray*}
&&\!\!\!\!\!\!\!\!\! \!\!\!\!\!\!\!\!\! q^{{1\over \delta}\bq_1\cdot \lambda(\bq_1-\bq_2)}
\prod_{\al=1}^r \qQ_{\al,1}^{q_{\al,2}+n_{\al,1}-2 q_{\al,1}}
\prod_{\beta=1}^r \qQ_{\beta,0}^{-1}\Big(\qQ_{\beta,0} (1-\prod_{\al=1}^r 
\qQ_{\al,1}^{-C_{\al,\beta}})^{-1}\Big)^{q_{\beta,1}+1}\\
&&\qquad\qquad\qquad=q^{-{1\over \delta}\sum_{\al,\beta}n_{\al,1}\lambda_{\al,\beta}}
\prod_{\al=1}^r \qQ_{\al,1}
\prod_{\beta=1}^r \qQ_{\beta,0}^{-1}\prod_{\al=1}^r \qQ_{\al,1}^{n_{\al,1}+1} \prod_{\beta=1}^r 
 \qQ_{\beta,2}^{-q_{\beta,1}-1} \prod_{\al=1}^r \qQ_{\al,1}^{q_{\al,2}}
\end{eqnarray*}
and the Lemma follows, since, as before, as $q_{\al,i+1}(\bm,\bn)=q_{\al,i}(\bm',\bn')$, where $\bm'$ 
are the new summation variables, and
the arguments are changed to $(\{\qQ_{\al,1}\},\{\qQ_{\al,2}\})$.
\end{proof}

Writing the solution
of the quantum $Q$-system as $\qQ_{\al,n}(\{\qQ_{\al,0}\},\{\qQ_{\al,1}\})$ 
to display its dependence on initial conditions,
we shall now use the following translational invariance property of the $Q$-system:
\begin{lemma}
For any solution of the quantum $Q$-system \eqref{qqsys}, we have:
$$\qQ_{\al,n}(\{\qQ_{\al,j}\},\{\qQ_{\al,j+1}\})=\qQ_{\al,n+j}(\{\qQ_{\al,0}\},\{\qQ_{\al,1}\})
\qquad (n\in\Z;j\in\Z_+;\al\in I_r).$$
\end{lemma}
The proof is similar to the case of $A_1$. This property allows to iterate the recursion relation \eqref{recuZsl}, 
which yields the factorization:
\begin{eqnarray*}
Z_{\bl;\bn}^{(k)}(\bqQ_0,\bqQ_1)&=&
q^{-{1\over \delta}\sum_{\al,\beta,i}n_{\al,i}\lambda_{\al,\beta}-{1\over 2\delta}\sum_\al \ell_\al\lambda_{\al,\al}}
q^{{1\over \delta}\sum_{\al<\beta}\lambda_{\al,\beta}} \\
&\times&\prod_{\al=1}^r \qQ_{\al,1}
\prod_{\beta=1}^r \qQ_{\beta,0}^{-1} \prod_{i=1}^k  \prod_{\al=1}^r\qQ_{\al,i}^{n_{\al,i}}
\prod_{\beta=1}^r
\big(\qQ_{\beta,k}\qQ_{\beta,k+1}^{-1}\big)^{\ell_\beta+1} 
\end{eqnarray*}
Using the commutation relations between the $\qQ$'s we finally arrive at:

\begin{thm}\label{zkfin}
\begin{eqnarray*}
Z_{\bl;\bn}^{(k)}(\bqQ_0,\bqQ_1)&=&
q^{-{1\over \delta}\sum_{\al,\beta,i}n_{\al,i}\lambda_{\al,\beta}-{1\over 2\delta}\sum_\al \ell_\al\lambda_{\al,\al}} \\
&\times&\left(\prod_{\al=1}^r \qQ_{\al,1}\qQ_{\al,0}^{-1}\right)
\left( \prod_{i=1}^k  \prod_{\al=1}^r\qQ_{\al,i}^{n_{\al,i}}
\right)\left(
\prod_{\beta=1}^r
\big(\qQ_{\beta,k}\qQ_{\beta,k+1}^{-1}\big)^{\ell_\beta+1} \right)
\end{eqnarray*}
\end{thm}
We have the subsequent obvious factorization:
\begin{cor}\label{corsl}
$$Z_{\bl;\bn_1,...,\bn_k}^{(k)}(\bqQ_0,\bqQ_1)=
Z_{0;\bn_1,...,\bn_j}^{(j)}(\bqQ_0,\bqQ_1) \, 
Z_{\ell;\bn_{j+1},...,\bn_k}^{(k-j)}(\bqQ_j,\bqQ_{j+1})$$
\end{cor}

\subsection{Proof of the $M=N$ identity}

Our task is to prove Theorem \ref{M=Nsl}. We need some preliminary Lemmas.

\begin{lemma}\label{comalbet}
We have the commutation relation:
$$ \qQ_{\al,-1}\qQ_{\gamma,1}=t^{2\lambda_{\al,\gamma}}\qQ_{\gamma,1}\qQ_{\al,-1} 
\quad {\rm for} \ {\rm all} \ \ \al\neq \gamma. $$
\end{lemma}
\begin{proof}
We note that  $\qQ_{\al,-1},\qQ_{\gamma,1}\in \by_{\vm}$ for the Motzkin path 
$\vm=(m_\beta)_{\beta\in I_r}$ with $m_\beta=-\delta_{\beta,\alpha}$, and apply Lemma \ref{comlem}.
\end{proof}

\begin{lemma}\label{positalpha}
The quantum $Q$-system solutions satisfy:
\begin{eqnarray}
\qQ_{\al,i}&\in& \Z_t[\{\qQ_{\beta,1},\qQ_{\beta,-1},\qQ_{\beta,0}^{\pm 1}\}]\label{propminusone}.
\end{eqnarray}
\end{lemma}
\begin{proof}
Using the Laurent property of quantum cluster algebras, we know that
$\qQ_{\al,i}\in \Z_t[\{\qQ_{\beta,0}^{\pm1},\qQ_{\beta,1}^{\pm1}\}_{\beta\in I_r}]$ and also 
$\qQ_{\al,i}\in\Z_t[\{\qQ_{\beta,-1}^{\pm1},\qQ_{\beta,0}^{\pm1}\}_{\beta\in I_r}]$. Equating the two 
Laurent polynomial expressions for $\qQ_{\al,i}$, and using the fact that
 $\qQ_{\beta,-1}$ is linear in $\qQ_{\beta,1}^{-1}$, there is an identification, monomial by monomial, 
 of terms of the form:
\begin{equation}\label{identimonom} \left(\prod_{\beta\in A}\qQ_{\beta,1}^{-m_\beta}
\prod_{\gamma\in B}\qQ_{\gamma,1}^{m_\gamma}\right)c = \left(\prod_{\beta\in A}\qQ_{\beta,-1}^{m_\beta}
\prod_{\gamma\in B}\qQ_{\gamma,-1}^{-m_\gamma}\right) d
\end{equation}
where $A,B\subset I_r$, $A\cap B=\emptyset$ and $c,d$ are Laurent polynomials in $\{\qQ_{\beta,0}\}_{\beta\in I_r}$. 
Using the commutation relations of Lemma \ref{comalbet}, which we can do because  $A$ and $B$ are disjoint sets, this is 
equivalent to:
$$  \left(\prod_{\gamma\in B}\qQ_{\gamma,1}^{m_\gamma}\right)\left(\prod_{\gamma\in B}\qQ_{\gamma,-1}^{m_\gamma}
\right)c = t^x\left(\prod_{\beta\in A}\qQ_{\beta,1}^{m_\beta}\right)\left(\prod_{\beta\in A}\qQ_{\beta,-1}^{m_\beta}\right) d$$
where $x$ is some integer.
Using the quantum $Q$-system relation 
\begin{equation}\label{qqone}
\qQ_{\beta,1}\qQ_{\beta,-1}=t^{-\lambda_{\beta,\beta}}\Big(\qQ_{\beta,0}^2-\prod_{\eta\sim \beta}\qQ_{\eta,0}\Big)
\end{equation}
this reduces to 
$$ P_B(\{\qQ_{\beta,0}\}_{\beta\in I_r})\, c = t^x P_A(\{\qQ_{\beta,0}\}_{\beta\in I_r}) \, d $$
where $P_A,P_B$ are polynomials of the $\qQ_{\beta,0}$'s of the form:
$$P_A=t^{x_A}\prod_{\beta\in A} (\qQ_{\beta,0}^2 -t^{x^A_\beta} \prod_{\eta\sim \beta} \qQ_{\eta,0}),
\qquad P_B=t^{x_B}\prod_{\gamma\in B} (\qQ_{\gamma,0}^2 -t^{x^B_\beta} \prod_{\eta\sim \gamma} 
\qQ_{\eta,0})$$
where $x_A$, $x^A_\beta$, $x_B$, $x^B_\gamma$ are integers.
These two polynomials $P_A$ and $P_B$ are clearly coprime, as $A\cap B=\emptyset$.
Thus, there exists a Laurent polynomial $e$ of $\{\qQ_{\beta,0}\}$, such that 
$c =t^x P_A \, e \quad {\rm and}\quad  d= P_B \, e$. Substituting this
into \eqref{identimonom} leads to
$$\left(\prod_{\beta\in A}\qQ_{\beta,1}^{-m_\beta}
\prod_{\gamma\in B}\qQ_{\gamma,1}^{m_\gamma}\right)c =t^z \left(
\prod_{\beta\in A}\qQ_{\beta,-1}^{m_\beta}\prod_{\gamma\in B}\qQ_{\gamma,1}^{m_\gamma}\right) e $$
for some integer $z$.
We conclude that $\qQ_{\al,i}$ may be written as a {\it polynomial} of 
$\{\qQ_{\beta,-1},\qQ_{\beta,1}\}_{\al\in I_r}$, with coefficients in $\Z_t[\{Q_{\al,0}^{\pm 1}\}]$, in the following form:
\begin{equation}\label{finqal}
\qQ_{\al,i}=\sum_{A\cup B=I_r, A\cap B =\emptyset
\atop m_\beta \in \Z_+} \left( \prod_{\beta\in A}\qQ_{\beta,-1}^{m_\beta}
\prod_{\gamma\in B}\qQ_{\gamma,1}^{m_\gamma}\right) \ \ c^{A,B}_{m_1,...,m_r}(\{\qQ_{\eta,0}^{\pm 1}\})
\end{equation}
where the sum is finite. This implies the Lemma.
\end{proof}

It follows immediately that
\begin{cor}
\begin{equation}
\prod_{\al,i}\qQ_{\al,i}^{n_{\al,i}}\in \Z_t[\{\qQ_{\beta,1},\qQ_{\beta,-1},\qQ_{\beta,0}^{\pm 1}\}]
\ {\rm for}\ {\rm finitely}\ {\rm many}\ n_{\al,i}\in \Z_+. \label{propzero}
\end{equation}
\end{cor}

The following ``evaluation map" will be useful in the next section:
\begin{defn}\label{defrho}
Let $\rho:\Z_{t}[\{\qQ_{\al,0}^{\pm 1}\}]((\{\qQ_{\al,1}^{-1}\}))\to \Z_{t}((\{\qQ_{\al,1}^{-1}\}))$ 
denote the  evaluation map:
$$\rho(f)=\left.\left(\prod_{\beta=1}^r \qQ_{\beta,1}\qQ_{\beta,0}^{-1}\times f\right)\right\vert_{\{Q_{\al,0}=1\}_{\al\in I_r}}$$
\end{defn}

We are now ready to evaluate the expressions of Lemma \ref{positalpha} using $\rho$ of Definition \ref{defrho}.
We first note the following:
\begin{lemma}\label{vaniminusone}
For any Laurent series $f\in \Z_{t}[\{\qQ_{\al,0}^{\pm 1}\}]((\{\qQ_{\al,1}^{-1}\}))$, 
and any $\gamma\in I_r$, we have
\begin{equation}
\rho(\qQ_{\gamma,-1}f ) =0 
\end{equation}
\end{lemma}
\begin{proof}
We write:
\begin{eqnarray*}
\left(\prod_{\beta=1}^r \qQ_{\beta,1}\qQ_{\beta,0}^{-1}\right) \qQ_{\gamma,-1}\, f&=&
t^x\left(\prod_{\beta\neq \gamma}  \qQ_{\beta,1}\qQ_{\beta,0}^{-1} \right)\qQ_{\gamma,0}^{-1}\qQ_{\gamma,1}\qQ_{\gamma,-1}\, f\\
&=&t^y\left(\prod_{\beta\neq \gamma}  \qQ_{\beta,1}\qQ_{\beta,0}^{-1} \right)\qQ_{\gamma,0}
(1-\prod_{\eta=1}^r\qQ_{\eta,0}^{-C_{\gamma,\eta}})\, f\\
&=&t^y\Big(1-\prod_{\eta=1}^r\qQ_{\eta,0}^{-C_{\gamma,\eta}}t^{\sum_{\beta\neq \eta}
C_{\gamma,\eta}\lambda_{\eta,\beta}}\Big)\left(\prod_{\beta\neq \gamma} 
 \qQ_{\beta,1}\qQ_{\beta,0}^{-1} \right)\qQ_{\gamma,0} \, f
\end{eqnarray*}
where $x$ and $y$ are integers.
The lemma follows by noting that $t^{\sum_{\beta\neq \eta}
C_{\gamma,\eta}\lambda_{\eta,\beta}}=1$ from $C\lambda =\delta I$.
\end{proof}

Thus, we have
\begin{cor}
\begin{eqnarray}
\rho(\qQ_{\al,i}) &\in& \Z_t[\{Q_{\beta,1}\}_{\beta\in I_r}] \label{propone} \\
\rho\left(\prod_{\al,i}\qQ_{\al,i}^{n_{\al,i}}\right)&\in & \Z_t[\{Q_{\beta,1}\}_{\beta\in I_r}] , 
\ {\rm for}\ {\rm finitely}\ {\rm many}\ n_{\al,i}\in \Z_+ \label{proptwo}
\end{eqnarray}
\end{cor}
\begin{proof}
From Lemma \ref{vaniminusone}, the only nontrivial contributions to $\rho(\qQ_{\al,i})$ 
come from terms with $A=\emptyset$ in \eqref{finqal}, and it follows
that $\rho(\qQ_{\al,i}) \in  \Z_t[\{Q_{\beta,1}\}_{\beta\in I_r}] $,
as the coefficients $c^{\emptyset,I_r}_{m_1,...,m_r}$ evaluate to elements of $\Z_t$, after commutation with 
$\prod_{\beta\in I_r} \qQ_{\beta,1}^{m_\beta}$. The same reasoning applies to $\prod_{\al,i}\qQ_{\al,i}^{n_{\al,i}}$.
\end{proof}

\begin{lemma}
The function $\qQ_{\al,i}^{-1}$ can be considered as a Laurent series in $\qQ_{\al,1}^{-1}$ 
with coefficients which are Laurent polynomials in the variables corresponding to the root labels $\beta\neq \al$, that is:
\begin{eqnarray}
\qQ_{\al,i}^{-1}&\in& \Z_t[\{\qQ_{\beta,1},\qQ_{\beta,-1}\}_{\beta\neq \al},\{\qQ_{\beta,0}^{\pm 1}\}]((\qQ_{\al,1}^{-1})).\label{propthree}
\end{eqnarray}
\end{lemma}
A similar lemma was proven in the commutative case in \cite{DFK}, and follows from the form of the 
$Q$-system equations. The proof in the quantum case follows the same reasoning.
\begin{proof}
We prove \eqref{propthree} by induction on $i$.
First, $\qQ_{\al,0}^{-1}$, as well as $\qQ_{\al,1}^{-1}$ satisfy the property \eqref{propthree}.
Assume the property holds for all $j\leq i$. Then
\begin{eqnarray*}
\qQ_{\al,i+1}^{-1}&=&t^{\lambda_{\al,\al}} \qQ_{\al,i-1}\left(1-\prod_{\beta} \qQ_{\beta,i}^{-C_{\al,\beta}} \right)^{-1}\qQ_{\al,i}^{-2}
=t^{-\lambda_{\al,\al}}\left(1-t^{-\delta}\prod_{\beta} \qQ_{\beta,i}^{-C_{\al,\beta}} \right)^{-1}\qQ_{\al,i}^{-2}\qQ_{\al,i-1}\\
&=&t^{-\lambda_{\al,\al}}\sum_{m\in \Z_+}t^{-m\delta} \qQ_{\al,i}^{-2(m+1)} \left(\prod_{\beta\neq\al}\qQ_{\beta,i}^{m|C_{\al,\beta}|}\right) \qQ_{\al,i-1}.
\end{eqnarray*}
By the recursion hypothesis, the term  $\qQ_{\al,i}^{-2(m+1)}$ is a Laurent series of $\qQ_{\al,1}^{-1}$, with
coefficients polynomial of $\qQ_{\beta,\pm 1}$ for $\beta\neq \al$, and Laurent polynomials
of the $\qQ_{\beta,0}$ for all $\beta$.
The remaining factor is a product of non-negative
powers of $\qQ$'s, and may be decomposed as in \eqref{finqal}, as a polynomial of the
variables $\qQ_{\beta,\pm 1}$ for all $\beta$, with coefficients that are Laurent polynomials 
of the $\qQ_{\beta,0}$ for all $\beta$.

We may now commute to the left each monomial of these polynomials that involve $\qQ_{\beta,\pm 1}$
for $\beta\neq \al$,
through the powers of $\qQ_{\al,1}^{-1}$ in the series. This is possible due to the $t$-commutation relations between 
all the terms, from Lemma \ref{comalbet}. Terms involving $\qQ_{\al,\pm 1}$ affect the Laurent series of $\qQ_{\al,1}^{-1}$
but respect the Laurent property. For instance the terms  $(\qQ_{\al,-1})^m$ for some $m\geq 0$ are
to be translated back in terms of $\qQ_{\al,1}$ as:
$$(\qQ_{\al,-1})^m=t^{-m\lambda_{\al,\al}} \left(\qQ_{\al,1}^{-1}\big(\qQ_{\al,0}^2-\prod_{\beta\sim \al} \qQ_{\beta,0}\big)
\right)^m$$
and simply contribute to the Laurent series of $\qQ_{\al,1}^{-1}$. The final result is therefore still a Laurent series
of $\qQ_{\al,1}^{-1}$ with coefficients in $\Z_t[\{\qQ_{\beta,1},\qQ_{\beta,-1}\}_{\beta\neq \al},\{\qQ_{\beta,0}^{\pm 1}\}]$
and the lemma follows.
\end{proof}

We  need the following generalization of Lemma \ref{vanilem}:

\begin{lemma}\label{techlem}
Let $P=\prod_j \qQ_{\al,j}^{a_j}\prod_{i,\beta\neq \al} \qQ_{\beta,i}^{b_{\beta,i}}$ 
for some $a_j\in \Z_+, \ b_{\beta,i}\in\Z$, then we have:
$$\CT_{\qQ_{\al,1}}\left(\rho(P)\right)=0$$
\end{lemma}
\begin{proof}
By the last property \eqref{propthree} of Lemma \ref{positalpha}, for each negative $b_{\beta,i}$, $\qQ_{\beta,i}^{b_{\beta,i}}$
is a Laurent series of $\qQ_{\beta ,1}^{-1}$
with coefficients in $\Z_t[\{\qQ_{\gamma,1},\qQ_{\gamma,-1}\}_{\gamma\neq \beta},\{\qQ_{\gamma,0}^{\pm 1}\}]$.
So $P$ has only non-negative powers of $\qQ_{\al,\pm 1}$, When we apply $\rho$, according to Lemma
\ref{positalpha}, all the $\qQ_{\beta,-1}$ are evaluated to $0$, so we are left with an expression
that is a Laurent series of some variables $\qQ_{\beta,1}^{-1}$ with $\beta\neq \al$, with coefficients
that are in particular polynomials of $\qQ_{\al,1}$, with valuation $1$ at least, due to the prefactor
in the definition of $\rho$. The constant term in $\qQ_{\al,1}$ therefore vanishes, and the lemma follows.

\end{proof}

Theorem \ref{M=Nsl} is equivalent to  the following statement.
We start from the expression \eqref{NsumSL} for the $N$-sum. 
\begin{lemma}
The sum over $\bm$ in \eqref{NsumSL} is unchanged if we restrict the sum to sets $\bm$ such that $q_{\al,j}\geq 0$, where $j=k,...,1$.
\end{lemma}
\begin{proof}
This is proved by induction as in the case of $A_1$. We assume the following induction hypothesis:

$(H_{j})$: {\it the result of the summation in \eqref{NsumSL}
remains unchanged if we restrict it so that $q_{\al,i}\geq 0$ 
for all $i=j+1,...,k$ and $\al\in I_r$.}

This clearly holds for $j=k-1$, as $q_{\al,k}=\ell_\al\in\Z_+$ 
for all $\al\in I_r$. Assume that it holds for some $j$,
and wish to prove it for $j-1$. Using the factorization property of Corollary \ref{corsl},
\begin{eqnarray}
\label{finsumsl}
\\
&&\!\!\!\!\!\!\!\! Z_{\bl;\bn}^{(k)}
=
q^{-{1\over \delta}\sum_{i=1}^j\sum_{\al,\beta} n_{\al,i}\lambda_{\al,\beta}-{1\over 2\delta}\sum_\al \ell_\al\lambda_{\al,\al}}
\prod_{\al=1}^r \qQ_{\al,1}
\qQ_{\al,0}^{-1} \prod_{1\leq \al\leq r\atop 1\leq i \leq j} \qQ_{\al,i}^{n_{\al,i}} \prod_{\beta=1}^r
\qQ_{\beta,j}\qQ_{\beta,j+1}^{-1}\nonumber \\
&& \ \ \ \ \  \times \sum_{\bm_{j+1},...,\bm_{k}\in \Z_+^r}q^{{1\over 2\delta}{\bq}_j\cdot \lambda \bq_j}
\prod_{\al=1}^r \qQ_{\al,j+1}^{-{q}_{\al,j}} \prod_{\beta=1}^r \qQ_{\beta,j}^{{q}_{\beta,j+1}}
\prod_{i=j+1}^k q^{{1\over 2\delta}({\bq}_i-{\bq}_{i+1})\cdot \lambda ({\bq}_i-{\bq}_{i+1})} 
\left[\begin{matrix}m_{\beta,i}+{q}_{\beta,i} \\ m_{\beta,i}\end{matrix}\right]_q \nonumber
\end{eqnarray}
We wish to prove that for {\it each} $\al\in I_r$ the contribution to the summation with $q_{\al,j}<0$
vanishes after the constant term in all the $\qQ_{\beta,1}$'s and the evaluation at all $\qQ_{\beta,0}=1$ are taken,
thereby establishing $(H_{j-1})$.
By $(H_j)$, we may restrict the summation to $q_{\beta,j+1},...,q_{\beta,k}\geq 0$ for all $\beta\in I_r$.
A generic term of the sum \eqref{finsumsl} at $\qQ_{\beta}=1$, $\beta\in I_r$
(apart from coefficients involving $t$) has the form
$$\rho\left(\prod_{1\leq \al\leq r\atop 1\leq i \leq j} \qQ_{\al,i}^{n_{\al,i}} \prod_{\beta=1}^r
\qQ_{\beta,j}\qQ_{\beta,j+1}^{-1}\prod_{\beta=1}^r \qQ_{\beta,j+1}^{-{q}_{\beta,j}} 
\prod_{\beta=1}^r \qQ_{\beta,j}^{{q}_{\beta,j+1}}\right).
$$

Due to Lemma \ref{techlem}, terms with $q_{\al,j}<0$ in \eqref{finsumsl} do not contribute
to the constant term of $Z_{\bl;\bn}^{(k)}(\{\qQ_{\al,0}\},\{\qQ_{\al,1}\})$. This holds for all $\al \in I_r$. Thus,
property $(H_{j-1})$ follows.
This completes the proof of the Lemma and therefore Theorem \ref{M=Nsl}.
\end{proof}

\subsection{Computing the graded multiplicities}

Now that we have proven Theorem \ref{M=Nsl}, we can use the constant term expression for $M_{\bl;\bn}^{(k)}(q)$. 
\begin{eqnarray*}
\label{MsumsL} M_{\bl;\bn}^{(k)}(q^{-1})&=&t^{\sum_{\al,\beta,i}n_{\al,i}\lambda_{\al,\beta}+
{1\over 2}(\sum_\al \ell_\al \lambda_{\al,\al}+\bn\cdot (\lambda\otimes A)\bn) } \\
&\times&\CT_{\bqQ_1}\left.\left( 
\prod_{\al=1}^r \qQ_{\al,1}\qQ_{\al,0}^{-1}\prod_{i=1}^k  \prod_{\al=1}^r\qQ_{\al,i}^{n_{\al,i}}
\prod_{\beta=1}^r
\big(\qQ_{\beta,k}\qQ_{\beta,k+1}^{-1}\big)^{\ell_\beta+1}
\right)\right|_{\bqQ_0=1}.
\end{eqnarray*}
Completing the sequence $\bn$ with zeros, and taking $k$ sufficiently large with respect to 
$\bn$, the constant term is independent of $k$, we have
\begin{thm}\label{howtocomputeM}
\begin{equation}\label{Msumslcomp}
M_{\bl,\bn}(q^{-1}) =t^{\sum_{\al,\beta,i}n_{\al,i}\lambda_{\al,\beta}+
{1\over 2}(\sum_\al \ell_\al \lambda_{\al,\al}+\bn\cdot (\lambda\otimes A)\bn) } 
\times \CT_{\bqQ_1}\
\rho\
\left(\prod_{i=1}^k  \prod_{\al=1}^r\qQ_{\al,i}^{n_{\al,i}}
\prod_{\beta=1}^r
z_{\beta}^{\ell_\beta+1}
\right)
\end{equation}
where
$$
z_\beta=\lim_{k\to \infty} z_{\beta,k}=\qQ_{\beta,0}\qQ_{\beta,1}^{-1}\prod_{k=1}^\infty
\left(1-\prod_{\al=1}^r (\qQ_{\al,k})^{-C_{\al,\beta}}\right)^{-1}.
$$
\end{thm}
\begin{proof}
We define $z_{\beta,k}=\qQ_{\beta,k}\qQ_{\beta,k+1}^{-1}$, and use the quantum $Q$-system to write
$z_{\beta,k}^{-1}z_{\beta,k-1}=1-\prod_{\al=1}^r (\qQ_{\al,k})^{-C_{\al,\beta}}$,
so that
$$
z_{\beta,k}=\qQ_{\beta,0}\qQ_{\beta,1}^{-1}\prod_{j=1}^{k}
\left(1-\prod_{\al=1}^r (\qQ_{\al,j})^{-C_{\al,\beta}}\right)^{-1}.
$$
Each $\qQ_{\al,j}$, expressed as a function of the initial data 
$\{\qQ_{\gamma,0},\qQ_{\gamma,1}\}_{\gamma\in  I_r}$
is a Laurent polynomial of degree $j$ in the variable $\qQ_{\al,1}$. 
This allows to expand $z_{\al,k}$ as
a formal power series of $\qQ_{\al,1}^{-1}$, with coefficients Laurent polynomials of the remaining variables,
and to extract the relevant constant term in the $\qQ_{\al,1}$'s.

When $k$ is sufficiently large with respect to $\bn$, the result is independent of $k$, and one may replace
$z_{\beta,k}$ by its limit when $k\to\infty$,
$z_\beta$ as in the statement of the Lemma. This last is a formal power series of $\qQ_{\beta,1}^{-1}$ 
with coefficients Laurent polynomial
of the remaining initial data.
\end{proof}

We have the following definitions, generalizing Definition \ref{unav}.

\begin{defn}\label{unavsl}
To any function $p\in\Z_t[\bqQ_{0}^{\pm1},\bqQ_{i}]_{i>0}$,
we associate the following Laurent polynomial of $t$:
$$ \mu_\bl(p) =\CT_{\bqQ_1}\rho\left(
\p \times
\prod_{\beta=1}^r
z_{\beta}^{\ell_\beta+1}
\right)\in\Z_t, $$
where $\bl=(\ell_1,...,\ell_r)$.
We use the following notation for the moments:
\begin{equation}\label{mulmsl}
\mu_{\bl,\bm}:=\mu_\bl\left(\prod_{\al=1}^rQ_{\al,1}^{m_\al}\right),
\end{equation}
where $\bm=(m_1,...,m_r)$.
\end{defn}
We have the following properties:

\begin{lemma}\label{faclemsl}
For any product of polynomials of the form $f g$, where $f= f(\bqQ_{0})\in\Z_t[\bqQ_0^{\pm1}]$ 
and $g\in\Z_t[\bqQ^{\pm1},\bqQ_i]_{i>0},$
we have
$$ \mu_\bl(f \ g)= f(\{ t^{-\sum_\beta \lambda_{\al,\beta}}\}_{\al\in I_r}) \mu_\bl( g). $$
\end{lemma}
\begin{proof}
This follows from
\begin{eqnarray*} 
\mu_\bl(\p)&=&\CT_{\bqQ_{1}}\rho\left(f(\bqQ_{0}) g
\times\prod_{\beta=1}^r z_{\beta}^{\ell_\beta+1}\right)\\
&=&\CT_{\bqQ_{1}}\left( f(\{t^{-\sum_\beta \lambda_{\al,\beta}}\qQ_{\al,0}\}_{\al\in I_r})  
\left(\prod_{\al=1}^r \qQ_{\al,1}\qQ_{\al,0}^{-1}\right) g
\prod_{\beta=1}^r z_{\beta}^{\ell_\beta+1}\right)\Bigg\vert_{\bqQ_{0}=1}\\
&=&f(\{ t^{-\sum_\beta \lambda_{\al,\beta}}\}_{\al\in I_r}) \mu_\bl(g)
\end{eqnarray*}
where the last expression is computed by (i) expressing $g$ as a Laurent polynomial of 
$(\{\qQ_{\al,0}\},\{\qQ_{\al,1}\})$,
(ii) moving all powers of $\qQ_{\al,0}$'s to the {\it left} (iii) taking the constant term in all the $\qQ_{\al,1}$
(iv) evaluating the expression at $\qQ_{\al,0}=1$. The left factor
$f(\{t^{-\sum_\beta \lambda_{\al,\beta}}\qQ_{\al,0}\}_{\al\in I_r})$ is eventually evaluated to  
$f(\{ t^{-\sum_\beta \lambda_{\al,\beta}}\}_{\al\in I_r})$, and the lemma follows.
\end{proof}

\begin{remark}\label{otherevalsl}
We may restate the result of Lemma \ref{faclemsl} as follows: the evaluation at all $\qQ_{\al,0}=1$ 
may be performed by commuting all the $\qQ_{\al,0}$'s of $\p$ to the left, 
and setting $\qQ_{\al,0}=t^{-\sum_\beta \lambda_{\al,\beta}}$. This will be instrumental
in our algebraic reformulation below.
\end{remark}

\begin{lemma}{(Polynomiality property)}\label{polpPsl}
For each (Laurent) polynomial  $\p\in \Z_t[\bqQ_0^{\pm 1},\bqQ_i]_{i>0}$
there exists a unique polynomial 
$P(\bqQ_{1})\in \Z_t[\bqQ_1]$ such that:
$$ \mu_\bl(\p)
=\mu_\bl(P(\bqQ_1)) $$
\end{lemma}
\begin{proof}
As explained in the proof of Lemma \ref{positalpha}, 
$\p$ may be expanded in a form similar to  \eqref{finqal}:
\begin{equation}\label{abcup}
\p= \sum_{A\cup B= I_r,\  A\cap B=\emptyset \atop
m_\al\in \Z_+} \left(\prod_{\al\in A}\qQ_{\al,-1}^{m_\al} \right)
\left(\prod_{\beta\in B} \qQ_{\beta,1}^{m_\beta} \right)\, 
c_{m_1,...,m_r}^{A,B}(\bqQ_{0})
\end{equation}
for finitely many non-vanishing Laurent polynomials $c_{m_1,...,m_r}^{A,B}$ of  $\bqQ_{0}$.
In fact, due to Lemma \ref{vaniminusone}, terms with $A\neq \emptyset$ vanish after the evaluation $\rho$. 
Applying $\mu_\bl$ to the Laurent expansion \eqref{abcup} of $\p$,
$$\mu_\bl(\p)=\mu_\bl \left(\sum_{m_1,...,m_r\in \Z_+}
c_{m_1,...,m_r}^{\emptyset, I_r}(\bqQ_0)\, 
\prod_{\beta=1}^r \qQ_{\beta,1}^{m_\beta} \right).$$
Finally using Lemma \ref{faclemsl} and Remark \ref{otherevalsl}, the lemma follows, with 
$$P(\{x_\al\}_{\al\in I_r})= \sum_{m_1,...,m_r\in \Z_+}
c_{m_1,...,m_r}^{\emptyset, I_r}(\{t^{-\sum_\beta \lambda_{\al,\beta}}\}_{\al\in I_r})\, 
\prod_{\beta=1}^r x_\beta^{m_\beta}. $$
\end{proof}

As in the $A_1$ case, we denote by $\varphi:\Z_t[\bqQ_0^{\pm1},\bqQ_i (i>0)]\to \Z_t[\bqQ_1]$ 
the map $\p\mapsto P$ given by the previous Lemma.
In view of Lemma \ref{polpPsl}, the $M$-sums are entirely determined by the coefficients 
$c_{m_1,...,m_r}^{\emptyset, I_r}$ and the numbers $\mu_{\bl,\bm}$ of eq.\eqref{mulmsl}.
The latter may be obtained by applying Lemma \ref{polpPsl} to
polynomials of the form: $\p=\prod_{\al=1}^r Q_{\al,m_{\al}}$ for some integers $m_\al\in\Z_+$. 
We may write the result as:
\begin{equation}\label{decompir}
\mu_\bl(\prod_{\al=1}^r Q_{\al,m_{\al}})=\mu_\bl(\varphi(\prod_{\al=1}^r Q_{\al,m_{\al}}))=
\sum_{j_1,...,j_\al \geq 0} c_{\bm,\bj}(t) \mu_{\bl,\bj} ,
\end{equation}
which, upon inversion, yields the values of $\mu_{\bl,\bj}$. 
Alternatively, the $\mu_{\bl,\bj}$ are determined by the graded $M$-sum formula
for integers $\bn$ such that $n_{\al,i}=\delta_{i,1} j_\al$. Explicitly:
\begin{equation*}
\mu_{\bl,\bj}=t^{
-{1\over 2}(\sum_\al \ell_\al \lambda_{\al,\al}+
\sum_{\al,\beta}\big\{ (j_\al+1) \lambda_{\al,\beta}(j_\beta+1)-\lambda_{\al,\beta}\big\}}
\, M_{\bl;\bn}(q^{-1}).
\end{equation*}

\subsection{An algebraic reformulation of the fusion product}

In the same spirit as for the $A_1$ case, we now present a purely algebraic formulation
of the fusion product of KR modules in the simply-laced case, 
using the solutions to the associated quantum $Q$-system.

More precisely, we use Remark \ref{otherevalsl} to reformulate the quantity $\mu_\bl(\p )$
as the matrix element of an algebra representation, which we denote by
$\langle 0,...,0\vert \p \vert \ell_1,...,\ell_r \rangle$.

Let ${\mathcal A}_+$ be the non-commutative algebra over $\Z[t,t^{-1}]$
generated by the variables $\qQ_{\al,0}$, $\qQ_{\al,0}^{-1}$, $\qQ_{\al,m}$, $\al\in I_r,m\in \Z_{>0}$, 
satisfying the quantum $Q$-system \eqref{qqsys} 
and the commutation relations \eqref{comq}. 

\begin{defn}\label{vacrepsl}
For $\bm=(m_1,...,m_r)\in (\Z_+)^r$, let 
$\langle m_1,...,m_r \vert$ denote the $\Z[t,t^{-1}]$--basis of a cyclic representation 
of the algebra ${\mathcal A}_+$
with a cyclic vector $\langle 0,0,...,0|$ such that
\begin{equation}\label{normavac}
\langle 0,0,...,0\vert \qQ_{\al,0} =t^{-{\sum_\beta \lambda_{\al,\beta}}}\, \langle 0,0,...,0\vert 
\end{equation}
and with basis vectors $\langle p_1,...,p_r\vert$, with $\bp=(p_1,...,p_r)\in (\Z_+)^r$, defined by:
\begin{eqnarray}
&&\sum_{\bp\in (\Z_+)^r} M_{\bp,\bn}(q^{-1})\, \langle p_1,...,p_r\vert \nonumber \\
&&\qquad \qquad 
=t^{-\sum_{\al,\beta}m_\al\lambda_{\al,\beta}-{1\over 2}(\sum_\al \ell_\al \lambda_{\al,\al}
+\sum_{\al,\beta}\min(m_\al,m_\beta)\lambda_{\al,\beta}) }
\langle 0,0,...,0\vert \prod_{\al=1}^r \qQ_{\al,m_\al} \label{defp}
\end{eqnarray}
where $\bn=(n_{\al,i})_{\al\in I_r, i\in \Z_+}$ with $n_{\al,i}=\delta_{i,m_\al}$ for all $\al,i$.
\end{defn}

Definition \ref{vacrepsl} gives an algebraic expression for the left evaluation at 
$\qQ_{\al,0}=t^{-\sum_\beta \lambda_{\al,\beta}}$ 
described in Lemma \ref{faclemsl} and Remark \ref{otherevalsl}.
Using the proofs of Lemmas \ref{vaniminusone} and \ref{polpPsl},
we have

\begin{lemma}
$$\langle 0,...,0 \vert \qQ_{\al,-1} =0\qquad (\al \in I_r)$$
\end{lemma}

\begin{lemma}\label{polypropsl}
With the $\qQ_{\al,i}$, $\al\in I_r, i\in \Z$, satisfying the quantum $Q$-system \eqref{qqsys} 
and the commutation relations \eqref{comq}, and
for an arbitrary function
$\p\in\Z_t[\bqQ_{0}^{\pm1},\{\bqQ_{\al,i}\}_{i>0}]$,
$$\langle 0,...,0 \vert \p 
=\langle 0,...,0 \vert \varphi(\p)(\bqQ_{1}). $$
\end{lemma}

Let $\vert \ell_1,...,\ell_r \rangle$, with $\bl=(\ell_1,...,\ell_r)\in (\Z_+)^r$ 
denote the dual vector basis to that of Def. \ref{vacrepsl}, namely
such that:
$$\langle m_1,...,m_r \vert \ell_1,...,\ell_r \rangle=\prod_{\al=1}^r
\delta_{m_\al,\ell_\al}.$$
We deduce the following main theorem, establishing the connection between our former evaluated constant term
and certain matrix elements of the present representation:

\begin{thm}
With the $\qQ_{\al,i}$, $\al\in I_r, i\in \Z$, satisfying the quantum $Q$-system \eqref{qqsys} 
and the commutation relations \eqref{comq}, and
for 
$\p\in \Z_t[\bqQ_{0}^{\pm1},\{\bqQ_{\al,i}\}_{i>0}]$
$$ \mu_\bl(\p)= \langle 0,...,0 \vert \p \vert \ell_1,...,\ell_r \rangle $$
in terms of the vector basis $\langle m_1,...,m_r \vert$ and its dual.
\end{thm}
\begin{proof}
By Lemmas \ref{polpPsl} and  \ref{polypropsl}, we only have to check that 
$\langle 0,...,0 \vert\prod_{\al=1}^r\qQ_{\al,1}^{j_\al} 
\vert \ell_1,...,\ell_r \rangle=\mu_{\bl,\bj}$ for all $\bl,\bj\in (\Z_+)^r$.
Applying eq.\eqref{defp} on the vector $\vert \ell_1,...,\ell_\al\rangle$ 
and comparing the result to the expression \eqref{Msumslcomp} for the graded $M$-sum, we get:
$$\langle 0,...,0 \vert\prod_{\al=1}^r \qQ_{\al,m_\al} \vert \ell_1,...,\ell_r\rangle= 
\mu_\bl(\prod_{\al=1}^r \qQ_{\al,m_\al}) $$
The theorem follows by expressing both terms in the variables $\qQ_{\al,1}$, $\al\in I_r$, via
Lemmas \ref{polpPsl} and \ref{polypropsl}.
\end{proof}

The results of this section may be summarized in the following:

\begin{thm}\label{oscirepsl} 
With Definitions \ref{vacrepsl},
the $M$-sum is given in terms of the solutions of the quantum 
$Q$-system by the following expectation value:
\begin{equation}\label{basch}
M_{\bl;\bn}(q^{-1})=q^{-{1\over \delta}\sum_{\al,\beta,i}n_{\al,i}\lambda_{\al,\beta}-
{1\over 2\delta}(\sum_\al \ell_\al \lambda_{\al,\al}+\bn\cdot (\lambda\otimes A)\bn) } 
\langle 0,0,...,0\vert \prod_{i=1}^k  \prod_{\al=1}^r\qQ_{\al,i}^{n_{\al,i}} \vert \ell_1,...,\ell_r\rangle.
\end{equation}
\end{thm}

We conclude with a few explicit examples.
\begin{example}
We consider the case of $A_2=s\ell_3$, with $\delta=3$,
$q=t^{-3}$. 
Let us compute the fusion product of KR modules: $KR_{1,1}\star KR_{1,1}\star KR_{2,2}$.
We consider the solutions of the $A_2$ quantum $Q$-system:
\begin{eqnarray*}
t^2 \qQ_{1,n+1}\qQ_{1,n-1}&=&\qQ_{1,n}^2-\qQ_{2,n}\\
t^2 \qQ_{2,n+1}\qQ_{2,n-1}&=&\qQ_{2,n}^2-\qQ_{1,n}, \qquad n\in \Z,
\end{eqnarray*}
in terms of the initial data 
$(\qQ_{1,0},\qQ_{2,0},\qQ_{1,1},\qQ_{2,1})$, and
with commutation relations
$$ \qQ_{1,n}\qQ_{1,n+1}=t^2 \qQ_{1,n+1}\qQ_{1,n}, \quad 
\qQ_{2,n}\qQ_{2,n+1}=t^2 \qQ_{2,n+1}\qQ_{2,n} \quad 
\qQ_{1,n}\qQ_{2,n+p}=t^p \qQ_{2,n+p}\qQ_{1,n}$$
where $|p|\leq 2$, $n\in \Z$.
The fusion product of KR modules: $KR_{1,1}\star KR_{1,1}\star KR_{2,2}$ is obtained
by computing:
$$\langle 0,0\vert \qQ_{1,1}^2 \qQ_{2,2}
=t^{-2}\langle 0,0\vert \qQ_{1,1}^2 (\qQ_{2,1}^2-\qQ_{1,1})\qQ_{2,0}^{-1}=
\langle 0,0\vert \big(t^7 \qQ_{1,1}^2\qQ_{2,1}^2 -t^{4} \qQ_{1,1}^3\big) $$
by commuting $\qQ_{2,0}^{-1}$ to the left and using 
$\langle 0,0\vert \qQ_{2,0}=t^{-3}\langle 0,0\vert$.
We express the two above terms in the $\langle m_1,m_2\vert$ basis via \eqref{defp}:
\begin{eqnarray*}
\langle 0,0\vert \qQ_{1,1}^2Q_{2,1}^2 &=&t^{-28}\langle 2,2\vert+t^{-24}\langle 3,0\vert +t^{-24}\langle 0,3\vert+t^{-23}(1+t^3)^2\langle 1,1\vert
+t^{-18}(1+t^6)\langle 0,0\vert \\
\langle 0,0\vert \qQ_{1,1}^3&=&t^{-21}\langle 3,0\vert +t^{-17}(1+t^3)\langle 1,1\vert+t^{-9} \langle 0,0\vert 
\end{eqnarray*}
where we have used
\begin{eqnarray*}
M_{(2,2);(1^2),(1^2)}&=&1\quad M_{(3,0);1^2,1^2}=t^3 \quad M_{(0,3);(1^2),(1^2)}=t^3\quad M_{(1,1);(1^2),(1^2)}=t^3(1+t^3)^2\\
M_{(0,0);(1^2),(1^2)}&=&t^6(1+t^6)\quad 
M_{(3,0);(1^3),()}=1 \quad M_{(1,1);(1^3),()}=t^3(1+t^3) \quad M_{(0,0);(1^3),()}=t^9 
\end{eqnarray*}
This gives:
$$\langle 0,0\vert \qQ_{1,1}^2 \qQ_{2,2}=t^{-21}\langle 2,2\vert+t^{-17}\langle 0,3\vert+t^{-16}(1+t^3)\langle 1,1\vert
+t^{-11}\langle 0,0\vert $$
Collecting all the prefactors, we finally get from \eqref{basch}:
$$M_{(2,2);(1^2),(2^1)}=1\quad M_{(0,3);(1^2),(2^1)}=t^3 \quad M_{(1,1);(1^2),(2^1)}=t^3(1+t^3) \quad M_{(0,0);(1^2),(2^1)}=t^6$$
and therefore:
$$ KR_{1,1}\star KR_{1,1}\star KR_{2,2} = V_{2,2} \oplus q^{-1} V_{0,3}\oplus q^{-1}(1+q^{-1})V_{1,1}\oplus q^{-2}V_{0,0}$$
\end{example}

\begin{example}
We consider the case of $D_4$. 
Let us compute the fusion product of KR modules: $KR_{1,1}\star KR_{3,3}$.
The $D_4$ quantum $Q$-system, with
$\qQ_{1,n}=T_n,\qQ_{2,n}=U_n,\qQ_{3,n}=V_n, \qQ_{4,n}=W_n$ 
and the Cartan matrix
$C=\begin{pmatrix}
2 & -1 & 0 & 0\\ 
-1 & 2 & -1 & -1 \\ 
0& -1& 2& 0\\ 0 &
 -1 & 0 & 2\end{pmatrix}$, takes the form:
\begin{eqnarray*}
t^4 T_{n+1}T_{n-1}&=& T_n^2 -U_n\qquad
t^4V_{n+1}V_{n-1}= V_n^2 -U_n\\
t^4W_{n+1}W_{n-1}&=& W_n^2 -U_n\qquad
t^8 U_{n+1}U_{n-1}=U_n^2 -T_nV_nW_n
\end{eqnarray*}
while the commutation relations corresponding to $\delta=4$ and
$\lambda=\begin{pmatrix}
4 & 4 & 2 & 2\\ 
4 & 8 & 4 & 4 \\ 
2 & 4 & 4 & 2\\ 
2 & 4 & 2 & 4\end{pmatrix}$ have the form, for all $n\in \Z$:
$$\begin{matrix}T_nT_{n+1}=t^4 T_{n+1}T_n & V_nV_{n+1}=t^4 V_{n+1}V_n &
W_nW_{n+1}=t^4 W_{n+1}W_n & U_nU_{n+1}=t^8 U_{n+1}U_n \\
T_n V_{n+p} = t^{2p}V_{n+p}T_n& T_n W_{n+p}=t^{2p}W_{n+p} T_n& 
T_n U_{n+p}= t^{4p}U_{n+p}T_n &\\
V_n W_{n+p} = t^{2p}W_{n+p}V_n& V_n U_{n+p}= t^{4p}U_{n+p}V_n & 
W_n U_{n+p}= t^{4p}U_{n+p}W_n & {}^{(p=0,\pm 1,\pm 2)}
\end{matrix}$$
We now compute:
$$T_1 V_3=t^{-12}\big(T_1V_1^3-(1+t^{4})T_1V_1U_1\big)V_0^{-2}
+t^{-8}T_1^2W_1U_0^{-1}-t^{-20}V_{-1}T_1U_1^2V_0^{-2}U_0^{-1}$$
Acting on the left vacuum, using Lemma \ref{vanilem} and commuting $V_0^{-2}$ and $U_0^{-1}$ to the left, we get:
$$\langle 0,0,0,0\vert T_1 V_3=\langle 0,0,0,0\vert \left(t^{40}T_1V_1^3-t^{32}(1+t^4)T_1V_1U_1+t^{24}T_1^2W_1\right)$$
We now express via \eqref{defp}:
\begin{eqnarray*}
\langle 0,0,0,0\vert T_1V_1^3&=& t^{-82}\langle 1,0,3,0\vert +t^{-68}\langle 2,0,0,1\vert +t^{-78}(1+t^4)\langle 1,1,1,0\vert\\
&&+t^{-72}(1+t^4+t^8)\langle 0,1,0,1\vert  +t^{-70}(1+t^4)(1+t^4+t^8)\langle 1,0,1,0\vert \\
&&+t^{-76}(1+t^4+t^8) \langle 0,0,2,1\vert +t^{-64}(1+t^4)(1+t^8)\langle 0,0,0,1\vert \\
\langle 0,0,0,0\vert T_1V_1U_1&=&t^{-70}\langle 1,1,1,0\vert +t^{-64}\langle 2,0,0,1\vert +t^{-64} \langle 0,0,2,1\vert
+ t^{-64}(1+t^4)\langle 0,1,0,1\vert  \\
&&+t^{-62} (1+t^4)^2\langle 1,0,1,0\vert +t^{-56}(1+t^4+t^8) \langle 0,0,0,1\vert  \\
\langle 0,0,0,0\vert T_1^2W_1&=& t^{-56}\langle 2,0,0,1\vert +t^{-52}\langle 0,1,0,1\vert  +t^{-50}(1+t^4)\langle 1,0,1,0\vert \\
&&+t^{-44}(1+t^4)\langle 0,0,0,1\vert
\end{eqnarray*}
We deduce that
$$ \langle 0,0,0,0\vert T_1 V_3= t^{-42}\langle 1,0,3,0\vert +t^{-36} \langle 0,0,2,1\vert $$
Finally, using \eqref{basch}, we get:
$$ M_{(1,0,3,0);(1^1)()(3^1)()}=1 \qquad M_{(0,0,2,1);(1^1)()(3^1)()}=t^4 $$
so that:
$$ KR_{1,1}\star KR_{3,3} =V_{1,0,3,0}\oplus q^{-1} V_{0,0,2,1} $$
\end{example}

\def\cprime{$'$} \def\cprime{$'$}

\end{document}